\documentclass[11pt, reqno]{amsart}  
\usepackage{amsmath, mathrsfs}
\usepackage{tikz}

\usepackage[all]{xy}
\SelectTips{cm}{}
\usepackage{amsthm, amsfonts}
\usepackage{geometry}
  \usepackage[overload]{empheq}
  
  \usepackage{tikz-cd}
  \usepackage{url}

  \usepackage{amssymb}
  
  \usepackage{graphicx}
  \usepackage{enumitem}
  \usepackage{color}
  \usepackage{enumitem}
  \usepackage[colorlinks=true]{hyperref}
  \hypersetup{urlcolor=cyan, citecolor=red, linkcolor=blue}
               
  \usepackage[capitalize]{cleveref}

  \usepackage{thmtools}

		\crefname{lemma}{Lemma}{Lemmas}
		\crefname{theorem}{Theorem}{Theorems}
		\crefname{prop}{Proposition}{Propositions}
		\crefname{cor}{Corollary}{Corollaries}
	
		\usepackage{xcolor}
\usepackage[most]{tcolorbox}
\usepackage{listings}

\definecolor{white}{rgb}{1,1,1}
\definecolor{mygreen}{rgb}{0,0.4,0}
\definecolor{light_gray}{rgb}{0.97,0.97,0.97}
\definecolor{mykey}{rgb}{0.117,0.403,0.713}

\tcbuselibrary{listings}
\newlength\inwd
\newtcblisting{pyin}[1][]{%
sharp corners,
enlarge left by=\inwd,
width=\linewidth-\inwd,
enhanced,
boxrule=0pt,
colback=light_gray,
listing only,
top=0pt,
bottom=0pt,
overlay={
  \node[
	anchor=north east,
	text width=\inwd,
	font=\footnotesize\ttfamily\color{mykey},
	inner ysep=1.5mm,
	inner xsep=0pt,
	outer sep=0pt
	] 
	at (frame.north west)
	{\refstepcounter{ipythcntr}\label{#1}In \theipythcntr:};
}
listing engine=listing,
listing options={
  aboveskip=1pt,
  belowskip=1pt,
  basicstyle=\footnotesize\ttfamily,
  language=Python,
  keywordstyle=\color{mykey},
  showstringspaces=false,
  stringstyle=\color{mygreen}
},
}

		\newtheoremstyle{1}
		{6pt} % Space above
		{0pt} % Space below
		{\itshape} % Body font
		{} % Indent amount
		{\bfseries} % Theorem head font
		{.} % Punctuation after theorem head
		{.5em} % Space after theorem head
		{} % Theorem head spec (can be left empty, meaning `normal')
		
		\newtheoremstyle{2}
		{6pt} % Space above
		{0pt} % Space below
		{} % Body font
		{} % Indent amount
		{\bfseries} % Theorem head font
		{.} % Punctuation after theorem head
		{.5em} % Space after theorem head
		{} % Theorem head spec (can be left empty, meaning `normal')
		
	\theoremstyle{1}
		\newtheorem{theorem}{Theorem}
		\newtheorem{lemma}[theorem]{Lemma}
		\newtheorem{prop}[theorem]{Proposition}
		\newtheorem{cor}[theorem]{Corollary}

    \theoremstyle{definition}
		\newtheorem{defn}[theorem]{Definition}
		\newtheorem{remark}[theorem]{Remark}
		
		\newtheorem{notation}[theorem]{Notation}

	\numberwithin{equation}{section}
		\numberwithin{lemma}{section}
		\numberwithin{theorem}{section}
    \numberwithin{prop}{section}
	\numberwithin{cor}{section}
		\numberwithin{defn}{section}
		\numberwithin{remark}{section}
		\numberwithin{example}{section}
		\numberwithin{assumption}{section}
    \numberwithin{conjecture}{section}
    \numberwithin{notation}{section}

        \newcommand{\F}{\mathbb{F}}
		\newcommand{\Q}{\mathbb{Q}}
		\newcommand{\Z}{\mathbb{Z}}
		\newcommand{\A}{\mathbb{A}}
		\newcommand{\Sym}{\textnormal{Sym}}

        \newcommand{\Ker}{\textnormal{Ker}}
		
		\newcommand{\Hom}{\textnormal{Hom}}
		\newcommand{\gsp}{\textnormal{GSp}}
		\newcommand{\Det}{\textnormal{det}}
		\newcommand{\flag}{\mathcal{F}l}
		\newcommand{\p}{\mathbb{P}}
		\newcommand{\KS}{\textnormal{KS}}
		\newcommand{\LL}{\mathcal{L}}
		\newcommand{\Ig}{\textnormal{Ig}}
    \newcommand{\GL}{\textnormal{GL}}
    \newcommand{\Isom}{\underline{\textnormal{Isom}}}

    \newcommand{\Rep}{\textnormal{Rep}}

    \newcommand{\dR}{\textnormal{dR}}
    
    \newcommand{\et}{\textnormal{\'et}}
    \newcommand{\tor}{\textnormal{tor}}
    \newcommand{\can}{\textnormal{can}}
    \newcommand{\sub}{\textnormal{sub}}
    \newcommand{\Sh}{\textnormal{Sh}}
    \newcommand{\Shbar}{\overline{\textnormal{Sh}}}
    \newcommand{\Fpbar}{\overline{\F}_p}
     \renewcommand{\H}{\textnormal{H}}
\begin{document}

\title{Theta operators and a generic entailment for $\text{GSp}_4$}
\author{Martin Ortiz}

\maketitle

% !TeX spellcheck = en_GB 

\begin{abstract}
We construct a new family of mod $p$ weight shifting differential operators on the Siegel threefold.
In particular, we construct one operator which generalizes the classical theta cycle, whose weight shift 
allows for maps between $p$-restricted weights, and which is generically injective on global sections. 
As an application we 
produce a generic entailment of Serre weights, i.e. any Hecke eigenform which is modular for a generic Serre weight 
in the lowest alcove is also modular for 
a Serre weight in one of the upper alcoves. 
The entailed Serre weight corresponds to a shadow weight of the lowest alcove Serre weight,
 in Herzig's conjectural description of $W(\overline{\rho})$. 
\end{abstract}

{\small \tableofcontents}
\renewcommand{\baselinestretch}{1}\normalsize

\section{Introduction}

The goal of this article is twofold. First, to study weight shifting differential operators 
over special fibers of Shimura varieties, i.e. mod $p$ theta operators, in the case of $\gsp_4/\Q$. 
Secondly, in \Cref{theoremA} we give the first example of a \textit{generic entailment} in the context of 
the weight part of Serre's conjecture. This shows that one can use 
theta operators to prove statements about Serre weights beyond the case of $\GL_2$.

The first theta operator was defined on mod $p$ modular forms by Katz \cite{katz-theta}, 
and it was then used by Edixhoven 
\cite{Edixhoven} 
to prove some cases of the weight part of Serre's conjecture. Concretely, given a non-ordinary mod $p$ eigenform
of weight $2 \le k \le p$,
repeated application of the theta operator (the so-called theta cycle, first appearing in \cite{Jochnowitz})
produces a form of weight $p-k+3$ with the same eigensystem up 
to a cyclotomic twist. More general theta operators 
have been defined by generalizing Katz's construction: Andreatta-Goren \cite{Andreatta-Goren} for Hilbert modular forms,
Yamauchi for $\gsp_4$ \cite{Yamauchi-1},
Eischen-Mantovan and their collaborators 
\cite{Eischen-Mantovan-1} \cite{Eischen-Mantovan-mu-ord} for type A and C PEL Shimura varieties,
and La Porta \cite{laporta} on strata of some unitary Shimura varieties. In this article we define a more extensive family of operators on the Shimura 
variety for $\gsp_4/\Q$, generalizing the phenomenon of the theta cycle to allow for a wider range of weight shifts.

Let $\Sh/\Z_p$ be the integral model of the Siegel threefold at hyperspecial level, and let $\Shbar$ be
its special fiber. 
One novelty is that we work over
the flag Shimura variety, which parametrizes flags on the Hodge bundle of
the Shimura variety: $\pi: \flag \to \Shbar$. For $(k,l) \in \Z^2$ one can define automorphic line bundles $\LL(k,l)$ over $\flag$
such that $\pi_* \LL(k,l)=\omega(k,l):=\Sym^{k-l} \omega \otimes \text{det}^l \omega$ are
precisely automorphic vector bundles on $\Shbar$. Thus, we can construct differential operators on $\Shbar$ from 
differential operators on $\flag$ by pushforward along $\pi$.
The operators
we define come in two different flavours. The first ones are 
\textit{basic theta operators} on $\flag$, whose construction is similar to Katz's.
First we define them on an open dense subset of $\flag$ characterized by admitting a canonical splitting of the 
symplectic flag, so that $\Omega^1_{\flag}$ splits into $4$ line bundles. 
Then we extend them to 
$\flag$ by multiplying by some appropriate number of Hasse invariants.
This produces $4$ theta operators 
$
\theta_i: \LL(\lambda) \to \LL(\lambda+\mu_i),
$
where 
$$
\mu_i=\begin{cases*}
(p+1,p-1), & i=1 \\
(2p,p-1), & i=2 \\
(2p,0), & i=3 \\
(p-1,0), & i=4
\end{cases*}
$$
with $\pi_* \theta_1$ already appearing in the literature, for instance in \cite{Yamauchi-1} and \cite{Eischen-Mantovan-1}.
See \Cref{weights-thetas} for the weight shifts taking into account the central character, which makes the operators Hecke equivariant. 
The weight increase of these operators is on the order of $p$,
which generically is insufficient for applications to the weight part of Serre's conjecture.
However, one can construct a second kind of operator with a smaller weight shift out of these basic operators. 
Using the property
that $\theta_1^p=H^p_1 \theta_3$ (where $H_1$ is the classical Hasse invariant), together with some control 
over when $H_1$ divides $\theta_1$, means that repeated application of $\theta_1$
exhibits a similar behaviour to Jochnowitz's theta cycle. Concretely, it produces an operator 
$$
\theta^1_{(k,l)}:=\theta^{p-k+1}_1 /H^{p-k+1}_1: \LL(k,l) \to \LL(2p-k+2,l) \;\;\; 1 \le k \le p
$$
whose weight increase is smaller.
We 
explain the significance of this map for the weight part of Serre's conjecture as outlined in \cite{Herzig1}
and \cite{Herzig-Tilouine}.
In characteristic zero due to Falting and Chai's BGG decomposition,
for $k \ge l \ge 3$
the non-Eisenstein Hecke eigensystems appearing in $\H^0(\Sh_{\overline{\Q}_p},\omega(k,l))$ are the same as the ones appearing in 
$\H^3_{\text{\'et}}(\Sh_{\overline{\Q}_p}, V(k-3,l-3)_{\overline{\Q}_p})$. Here $V(a,b)$ is the $\overline{\Q}_p$-local system
associated to the representation of $\gsp_4$ of highest weight $(a,b)$ with the same name,
the so-called dual Weyl module of weight $(a,b)$.
One can define dual Weyl modules integrally, but in general their mod $p$ reductions will be reducible.
Denote by $F(a,b)$ the socle of $V(a,b)_{\overline{\F}_p}$ as a $\gsp_4(\F_p)$-representation. Then 
 the irreducible $\gsp_4(\F_p)$-representations over $\overline{\F}_p$, i.e. \textit{Serre weights},
are all of the form $F(a,b)$ for \textit{$p$-restricted weights}: $\{0 \le a-b, b < p\}$. 
 Given a mod $p$ Hecke eigensystem $\mathfrak{m} \subseteq \mathbb{T}$ in the spherical Hecke algebra $\mathbb{T}$
 which acts on the \'etale cohomology of $\Sh$,
  the weight part of Serre's conjecture studies 
 the set of modular Serre weights
 $
 W(\overline{\rho}_{\mathfrak{m}}):=\{F(a,b): \H^*_{\et}(\Sh_{\overline{\Q}_p},F(a,b))_{\mathfrak{m}} \neq 0 \}.
 $ 
For a slightly more restrictive notion than $p$-restricted weights 
one can lift mod $p$ coherent cohomology to characteristic $0$
by a result from \cite{Lan-Suh-non-compact} and \cite{alexandre}. On the other hand, when 
$\mathfrak{m}$ is non-Eisenstein and \textit{generic} (i.e. there exists an auxiliary prime $l \neq p$ such that $\mathfrak{m}$
is generic at $l$ in the sense of \cite[Def 1.1]{Hamann-Lee}), then $\H^3_{\et}(\Sh_{\overline{\Q}_p},V(a,b)_{\Fpbar})_{\mathfrak{m}}$
lifts to characteristic $0$.
Then under these hypotheses on the weight and $\mathfrak{m}$,
the BGG decomposition in characteristic $0$ implies that
 $$
 \H^0(\Sh_{\overline{\F}_p},\omega(k,l))_{\mathfrak{m}} \neq 0 \iff  
\H^3_{\et}(\Sh_{\overline{\Q}_p},V(k-3,l-3)_{\overline{\F}_p})_{\mathfrak{m}} \neq 0.
$$
Crucially, 
in \Cref{geometric-entailment} we prove that for
$0 \le k \le p-1$ the map 
$\theta^1_{(k,l)}$ is injective on $\H^0(\Shbar,\omega(k,l))$. We do this by proving the injectivity of
the simpler operator $\theta_3$, which implies the injectivity of $\theta^1_{(k,l)}$ due to the relation $\theta^p_1=H^p_1\theta_3$. 
We prove the latter relation using Serre--Tate coordinates. 
 After translating to \'etale cohomology 
and using the decomposition of dual Weyl modules into Serre weights,
we get the following result. The $p$-restricted region $X_1(T)$ can be divided into four $\rho$-shifted alcoves 
$C_{i}$ for $i=0,1,2,3$. Given $\lambda \in C_0$ let $\lambda_i \in C_i$ be their corresponding affine Weyl reflections. 
\begin{theorem} \label{theoremA}
  (\Cref{entailment})
  Let $\lambda_0=(a,b) \in X^*(T)$ satisfying $a \ge b \ge 1, a+b<p-3$ and let $\mathfrak{m} \subset \mathbb{T}$
  be a generic non-Eisenstein eigensystem such that $F(\lambda_0) \in W(\overline{\rho}_{\mathfrak{m}})$.
  Then 
  $$
  F(\lambda_1) \in W(\overline{\rho}_{\mathfrak{m}}) \;\textnormal{ or } \; F(\lambda_2) \in W(\overline{\rho}_{\mathfrak{m}}),
  $$ where 
  $\lambda_1=(p-b-3,p-a-3) \in C_1$ or $\lambda_2=(p+b-1,p-a-3) \in C_2$. 
\end{theorem}
To sketch the proof of \Cref{theoremA}, the key fact is that $F(\lambda_0)$ is a Jordan--Holder factor of $V(\lambda_1)_{\Fpbar}$,
while $V(\lambda_2)_{\Fpbar}$ has only
$F(\lambda_1)$ and $F(\lambda_2)$ as factors. Then 
\begin{align*}
& \H^3_{\et}(\Sh_{\overline{\Q}_p},F(\lambda_0))_{\mathfrak{m}} \neq 0 \implies 
\H^3_{\et}(\Sh_{\overline{\Q}_p}, V(\lambda_1)_{\Fpbar})_{\mathfrak{m}} \neq 0
\implies \H^0(\Shbar,\omega(\lambda_1+(3,3)))_{\mathfrak{m}} \neq 0 \\
& \implies \H^0(\Shbar,\omega(\lambda_2+(3,3)))_{\mathfrak{m}} \neq 0
 \implies \H^3_{\et}(\Sh_{\overline{\Q}_p}, V(\lambda_2)_{\Fpbar})_{\mathfrak{m}} \neq 0,
\end{align*}
where the second to last arrow is implied by the injectivity of $\theta^1_{\lambda_1+(3,3)}$.
In \cite{Gee-Herzig-Savitt} they introduced the terminology that
$F(\lambda_0)$ \textit{entails} $F(\lambda_1)$
or $F(\lambda_2)$ to denote this phenomenon in which for every $\mathfrak{m}$\footnote{We allow ourselves to restrict 
to generic non-Eisenstein $\mathfrak{m}$.} being modular for some Serre weight implies being modular 
for some other Serre weights. 
This is the first instance of a generic entailment in the literature, where generic means 
that it holds for $\lambda_0$ sufficiently away from the walls of the lowest alcove 
(the region $\{a \ge b \ge 0, a+b <p-3\}$). We also produce examples of 
non-generic entailments using $\theta_1$ and $H_1$ in \Cref{non-generic}.

We explain the relation of this entailment to the structure of the Breuil--Mezard cycles. The Breuil--Mezard conjecture (\cite{Breuil-Mezard},
\cite{EG-stack} in this formulation)
predicts the existence of top dimensional cycles $\mathcal{Z}_{V} \in \text{Ch}_{\text{top}}(\mathcal{X}_{\text{red}})$ inside 
the reduced moduli stack of $4$ dimensional symplectic $(\phi,\Gamma)$-modules for each 
$V \in K^0(\Rep_{\overline{\F}_p}(\gsp_4(\F_p))$. 
These should satisfy that
$\mathcal{Z}_{V(\lambda)}$ is equal to the class of $\mathcal{X}^{\lambda+\rho}_{\F_p}$: the special fiber of the substack
of crystalline Galois representations 
with Hodge-Tate weights $\lambda+\rho$. Assuming this conjecture \cite{Gee-Herzig-Savitt} formulate a general version of 
the weight part of Serre's conjecture: $W(\overline{\rho}_{\mathfrak{m}})=\{ \sigma :   
\mathcal{Z}_{\sigma}(\overline{\rho}_{\mathfrak{m}})\neq 0 \}$. Assuming the Breuil--Mezard conjecture, this version of 
the weight part of Serre's conjecture, and strong enough globalization results, then \Cref{theoremA} 
implies that
$
\mathcal{Z}_{F(\lambda_0)} \subset \mathcal{Z}_{V(\lambda_2)}=\mathcal{Z}_{F(\lambda_2)}+\mathcal{Z}_{F(\lambda_1)}.
$
Through some
computations on local models of Galois deformation rings one can see that it must in fact be that
$\mathcal{Z}_{F(\lambda_0)} \subset \mathcal{Z}_{F(\lambda_2)}$, and that 
$F(\lambda_0) \in W(\overline{\rho}_{\mathfrak{m}}) \implies F(\lambda_2) \in W(\overline{\rho}_{\mathfrak{m}})$
should be the only generic entailment for $\gsp_4/\Q_p$. By a similar argument, there should not
be a generic entailment for 
$\GL_{3}/\Q_p$, so in some sense this constitutes the simplest example of a generic entailment. 
See also forthcoming work of Le Hung-Lin \cite{Le-Hung-Lin}
where they 
determine generic Breuil--Mezard cycles in terms of their irreducible components (of the reduced Emerton-Gee stack) 
for low rank groups including $\gsp_4/\Q_p$,
building on the approach of \cite{Le-Hung-Feng}.

We make some comments on possible generalizations. In forthcoming work \cite{Hodge-paper} we will generalize the construction 
of the basic theta operators to Hodge type Shimura varieties, as well as providing a more conceptual way 
to construct them, and we will further 
the study of their properties. We will also generalize a class of maps containing $\theta^1_{(k,l)}$, which we 
denote by \textit{theta linkage maps}. We have decided to keep the original construction of 
the basic theta operators and of $\theta^1_{(k,l)}$, as it 
is more similar to the construction of previous theta operators. 
Using this more general theory, or the techniques of this paper, 
we can also produce a generic entailment for a unitary Shimura variety with $G_{\Q_p}=\GL_4 \times \mathbb{G}_m$.
However, the main limitation of our method concerns the decomposition of dual Weyl modules into irreducible constituents, which 
at the moment
prevents us from producing examples of generic entailments for groups of arbitrarily high rank.

The organization of this article is as follows. In \Cref{Section1} we recall the geometric and representation 
theoretic setup 
for the Siegel threefold. In \Cref{Section2} we define the basic theta operators on an open dense subset $U \subset \flag$.
In \Cref{Section3} we extend them to $\flag$ and we prove some key properties about them
in the general spirit of previous literature. This includes the relation to the 
operators in \cite{Eischen-Mantovan-1}, culminating in \Cref{geometric-entailment}.
In \Cref{Section5} we spell out the proof of the generic entailment, and we give some examples of non-generic entailments. 

\textbf{Acknowledgments.} First and foremost I want to thank my advisor George Boxer for his continued support,  
and great generosity with his time and mathematical knowledge, as well as his opportune comments.
 I am grateful to my other advisor Payman Kassaei for introducing
me to this area and his interest on my work. I thank Fred Diamond and Brandon Levin for conversations about my work.
I also thank Lorenzo La Porta for all of our conversations and for comments on an earlier draft. Finally, I would like to 
thank my peers and friends at Imperial College. 
 This work was supported by the Engineering and Physical Sciences Research Council [EP/S021590/1]. 
The EPSRC Centre for Doctoral Training in Geometry and Number Theory 
(The London School of Geometry and Number Theory), University College London, King's College London and 
Imperial College London. 

\section{Set up} \label{Section1}

Let $p$ be a prime, $G=\gsp_4$ as a reductive group over $\Z$. Our convention is that $G$ is the group
of symplectic similitudes of $\Z^4$ with its symplectic form given by the matrix 
$$
J=\begin{pmatrix}
0 & S \\
-S & 0 
\end{pmatrix}
$$
with $S=\begin{pmatrix}
  0 & 1 \\
  1 & 0 
  \end{pmatrix}$.
Let $\A^{\infty}$ be the finite adeles of $\Q$, we fix a level $K=K^pK_p \subset G(\A^{\infty})$ hyperspecial at $p$,
 in particular its component at $p$
 satisfies that $K_p=G(\Z_p)$. We also assume that $K$ is neat, so that
one can define a smooth quasi-projective integral model $\text{Sh}_{K}/\Z_p$ of the 
 Siegel Shimura variety associated to $(G,\mathcal{H}_2,K)$ as follows. 
 \begin{defn} \label{Shimura-variety}
	Let $F: \text{Sch}/\Z_p \to \text{Sets}$ be
	the functor defined by $F(S)$ being the set of tuples $(A,\lambda, \eta)$ up to equivalence, where
	\begin{itemize}
	\vspace{-0.6em}
    \item $A/S$ is an abelian scheme of relative dimension $2$.
    \item $\lambda: A \to A^{\vee}$ is a prime to $p$ quasi-isogeny such that there exists an 
    integer $N \ge 1$ with $N\lambda$ a polarization, i.e. a prime to $p$ quasi-polarization. 
    \item $\eta$ is a \textit{rational} $K$-level  structure of $A$ as in \cite[1.3.8]{Lan-thesis}.
	\end{itemize}
	  Two tuples
	 $(A,\lambda, \eta)$ and $(A',\lambda',\eta')$ are equivalent if there exists a prime to $p$ quasi-isogeny
	  $\phi: A \to A'$ commuting with $\lambda$ and $\lambda'$ up to a $\Z^{\times,>0}_{(p)}$ constant, and such that the pullback of
	  $\eta'$ under $\phi$ is $\eta$. This functor is represented by a smooth quasi-projective scheme 
	  $\text{Sh}_K$ over $\Z_p$. We will mostly drop the level from the notation. Let $\Shbar=\Sh \otimes \overline{\F}_p$.
	 
 \end{defn}
  There is also a moduli description of $\text{Sh}_K$ in terms of isomorphism classes of abelian varieties, 
  which is more convenient when doing deformation theory. 
 \begin{prop} \cite[1.4.3.3]{Lan-thesis} \label{Sh-moduli-isomorphisms}
 $\textnormal{Sh}_K$ also represents the following moduli problem $\textnormal{Sch}/\Z_p \to \textnormal{Sets}$
  whose $S$-points are tuples $(A,\lambda,\eta)$ up to isomorphism
 \begin{itemize}
 \item $A/S$ an abelian scheme of relative dimension $2$.
 \item $\lambda : A \to A^\vee$ a prime to $p$ polarization.
 \item $\eta$ is an \textit{integral} $K$-level structure of $(A,\lambda)$ as in \cite[1.3.7.8]{Lan-thesis}.
 \end{itemize}
 An isomorphism between $(A,\lambda, \eta)$ and $(A',\lambda',\eta')$ is an isomorphism $f: A \to A'$
 compatible with polarizations, and such that it sends $\eta'$ to $\eta$. 
 \end{prop}

  Let $A/\Sh$ be the universal abelian variety with identity $e$, its Hodge bundle is defined by
  $$
  \omega \coloneqq e^* \Omega^1_{A/\Sh}.
  $$
  Our notation for weights is as follows: for $(k,l) \in \Z^2$ satisfying $k\ge l$,
  the automorphic vector bundle on $\Sh$ of weight $(k,l)$
  is
  $$
  \omega(k,l)\coloneqq\Sym^{k-l} \omega \otimes \Det^{l} \omega.
  $$
  Then the space of mod $p$ Siegel modular forms of genus $2$ of weight $(k,l)$ is $\H^0(\Shbar,\omega(k,l))$.
  As a useful auxiliary space we introduce the flag Shimura variety.

  \begin{defn}[Flag Shimura variety]
   We define $\flag$ to be the flag space of $\omega$, with the convention 
   that 
   $$
   \flag=\underline{\textnormal{Proj}}(\text{Sym}^{\bullet} \omega) \xrightarrow{\pi} \Sh.
   $$
   It is smooth proper over $\Sh$, with its fibers isomorphic to $\p^1$, and it comes equipped with a line bundle 
   $\mathcal{L} \subset \pi^* \omega$. For $(k,l) \in \Z^2$ we define
   the line bundle $\mathcal{L}(k,l)$ on $\flag$ by
   $$
   \mathcal{L}(k,l)\coloneqq (\omega/\mathcal{L})^k \otimes \mathcal{L}^l.
   $$
    
  \end{defn}
  We record some properties about the flag Shimura variety, let $\pi: \flag \to \Sh$. 
  \begin{lemma} \label{basic-lemma-flag}
  \begin{enumerate}
   \item For $(k,l) \in \Z^2$, we have $\pi_* \LL(k,l)=\omega(k,l)$. In particular 
   the space of global sections of $\LL(k,l)$ is the space of Siegel modular forms of weight $(k,l)$
   \cite[\href{https://stacks.math.columbia.edu/tag/01XX}{Tag 01XX}]{stacks-project}.
   \item There is a canonical isomorphism $\Omega^1_{\flag/\Sh}=\mathcal{L}(-1,1)$.
   \item $\det \omega=\mathcal{L} \otimes (\omega/\mathcal{L})$.

  \end{enumerate}
  \begin{proof}
  For $2)$ on each affine $U \subset \Sh$ such that $\omega_{U}$ is trivial the isomorphism reduces to
  the canonical isomorphism $\Omega^1_{\mathbb{P}^1_U/U}=\mathcal{O}(-2)_{U}=(\LL \otimes (\omega/\LL)^{-1})_{U}$, 
  and we can check that these isomorphisms glue.  
  \end{proof}
  \end{lemma}

  The degeneration of the Hodge-de-Rham spectral sequence for abelian varieties yields the exact sequence
  $$
  0 \to \omega \to \H^1_{\text{dR}}(A/\Sh) \to  \omega^{\vee}_{A^\vee} \to 0, 
  $$
  and $\omega_{A^\vee} \xrightarrow{d\lambda} \omega_{A}$ is an isomorphism
  since $\lambda$ is prime to $p$. We will usually make this identification, whose only role will appear 
  when considering the Hecke action. 
  Let $H$ denote $\H^1_{\text{dR}}(A/\Sh)$ throughout the rest of the paper.
  This way $H$ is a locally free sheaf, and it carries a symplectic pairing
  $\langle, \rangle$ induced by 
  the quasi-polarization and the canonical pairing $\H^1_{\dR}(A/\Sh) \otimes \H^1_{\dR}(A^{\vee}/\Sh) \to \mathcal{O}_{\Sh}$.
   In characteristic $p$ it also comes with Verschiebung $V: H \to H^{(p)}$ and Frobenius 
  $F: \H^{(p)} 
  \to \H$
  maps
  induced by the respective maps on $A/\Shbar$. Then $\omega$ is a maximal isotropic subspace for the symplectic
  pairing, and $V, F$ satisfy 
  $$
  \langle Vx, y \rangle =\langle x,Fy \rangle ^{(p)}
  $$
  for $x \in H, y \in H^{(p)}$. It is also equipped with the Gauss-Manin connection. 

  In general, let $Y/S/T$  be schemes, with $Y/S$ and $S/T$ smooth. 
 The Gauss-Manin connection is a natural connection \cite{Katz-Oda}
 $$
 \nabla_{Y/S/T} : \H^1_{\text{dR}}(Y/S) \to \H^1_{\text{dR}}(Y/S) \otimes_{\mathcal{O}_S} \Omega^1_{S/T}.
 $$
 It is compatible with base change in the following sense. 
 \begin{lemma}(Functoriality of $\nabla$) \label{GM-base-change}
 Let $Y/S/T$ as before, $f:S' \to S$ a morphism of $T$-schemes, and let 
 $Y'\coloneqq Y \times_S S'$. Let $e$ be a local section of $\H^1_{\text{dR}}(Y/S)$, we can pull it back to a
 local section $f^* e$ of $\H^1_{\text{dR}}(Y'/S')=f^*\H^1_{\text{dR}}(Y/S)$. Then
 $$
\nabla_{Y'/S'/T}(f^* e)=\textnormal{id} \otimes \textnormal{df}(f^* \nabla_{Y/S/T}(e))
 $$
 where $\textnormal{df}: f^* \Omega^1_{S/T} \to \Omega^1_{S'/T}$ is the differential of $f$. That is,
 the diagram 
 $$
\begin{tikzcd}
\H^1_{\text{dR}}(Y'/S') \arrow["\nabla_{Y'/S'/T}",r] & \H^1_{\text{dR}}(Y'/S') \otimes \Omega^1_{S'/T} \\
\H^1_{\text{dR}}(Y/S) \arrow["f^*",u] \arrow["\nabla_{Y/S/T}",r] & \H^1_{\text{dR}}(Y/S) \otimes \Omega^1_{S/T}
\arrow["f^* \circ \textnormal{df}",u]
\end{tikzcd}
 $$
 commutes when interpreted at the level of local sections. 

 \end{lemma}

 Consider $\nabla$ with respect to $A/\Sh/\Z_p$. It respects the symplectic structure in the sense that 
 $\langle \nabla v, w \rangle + \langle v, \nabla w\rangle=d\langle v,w \rangle$ for $v,w \in H$. 
 By restricting $\nabla$ to $\omega$ and then projecting along the map $H \to \omega^{\vee}_{A^\vee} $
 we obtain the Kodaira--Spencer morphism 
 $$
\KS: \omega_{A} \otimes \omega_{A^{\vee}} \to \Omega^1_{\Sh/\Z_p}.
 $$
 After applying the isomorphism $d\lambda: \omega_{A^{\vee}} \to \omega_{A}$ one can see that the map factors through
 an isomorphism
  $\text{ks}: \text{Sym}^2 \omega
\cong \Omega^1_{\Sh/\Z_p}$ also denoted by Kodaira--Spencer.

Now we define the two Hasse invariants on $\flag/\overline{\F}_p$. Recall that for $L$ a line bundle on a characteristic $p$
  scheme $X$ there is a canonical isomorphism $L^{(p)}\coloneqq\text{Frob}^{*}_X L \cong L^{\otimes p}$.
  \begin{defn}(Hasse invariants)
   On $\flag$ consider the Verschiebung map $V: \omega \to \omega^{(p)}$.
   \begin{enumerate}
\item By taking determinants of $V$ we obtain a map $\det V: \det \omega \to (\text{det} \omega)^{(p)}=\text{det}^p \omega$,
      which produces a section 
	  $$
      H_1 \in H^0(\flag,\mathcal{L}(p-1,p-1)).
	  $$
	  We say that $H_1$ is the first Hasse invariant of $\flag$.
\item Consider the composite map $\mathcal{L} \hookrightarrow \omega \xrightarrow{V}  \omega^{(p)}
\to (\omega/\mathcal{L})^{(p)}$. This produces the second Hasse invariant
$$
H_2 \in \H^0(\flag, \mathcal{L}(p,-1)).
$$
\item Let $D_i$ be the locus where $H_i$ vanishes, and let $U_i$ be its complement in $\flag$. 
   \end{enumerate}
  \end{defn}

  The first Hasse invariant is the pullback of the classical Hasse invariant
  $H \in \H^0(\Shbar, \det^{p-1} \omega)$, while $H_2$
  lives naturally on the flag variety. These are precisely the Hasse invariants 
  considered in much more generality in \cite{Hasse-flag}. Moreover, we will sometimes use the $p$-rank $1$ Ekedahl-Oort strata 
  $\Shbar^{=1} \subseteq \Shbar$ defined by $V: \omega \to \omega^{(p)}$ having rank $1$, and 
  $V^2 \neq 0$.

  \begin{notation} \label{notation-local-computation}
  We explain our notation for local computations on $\flag$,
 which will be used throughout. Let $W \subset \Shbar$
  be an open subset of $\Shbar$ such that $\omega|W$ is free with basis $\{ e_1, e_2 \}$.
   On the pullback of
  $\omega$ to $\flag$ we use the pullback of this basis. Then on $W$
  we can identify $\flag$ with $\p^1_W$ and $\mathcal{L}$ with $\mathcal{O}(-1)$. We denote the coordinates
  of $\p^1=\text{Proj}(\Z[x,y])$ by $x,y$, and we consider $\mathcal{L}$ embedded in $\mathcal{O}^2 \cong \omega$
  by 
  $$
  \mathcal{L} \cong \mathcal{O}(-1) \xrightarrow{(x,y)} \mathcal{O}^2 \cong \omega.
  $$
  On the affine chart $\{x \neq 0\}$ $\mathcal{L}$ is free with basis $e_1 + \frac{y}{x}e_2$. We will
  denote $\frac{y}{x}$ by $T$ and refer to the associated chart 
  as $\mathbb{A}^1_{W} \subset \p^1_{W}$, where $\omega/\mathcal{L}$
  is free with basis $e_2$. Thus, in this chart we will use 
  $$
  e^k_2 (e_1+Te_2)^l=e^{k-l}_2 (e_1 \wedge e_2)^l
  $$
  as a local basis for $\LL(k,l)$.
  Let 
  $$
  V=\begin{pmatrix}
   a & b \\
   c & d 
  \end{pmatrix}
  $$
  be the matrix of $V: \omega \to \omega^{(p)}$ with respect to the local basis, e.g.
  $V(e_1)=a e^{(p)}_1 + c e^{(p)}_2$. Then 
  $$
  H_1=(ad-bc)(e_1 \wedge e_2)^{p-1}.
  $$
  A basis for $\mathcal{L}^{(p)}$ on $\A^1_W$ is $e^{(p)}_1+T^p e^{(p)}_2$, and 
  $$
  V(e_1+Te_2)=(a+bT)e^{(p)}_1+(c+dT)e^{(p)}_2=(c+dT-aT^p-bT^{p+1})e^{p}_2 \in (\omega/\LL)^p,
  $$ so that
  \begin{align*} 
  H_2=&(c+dT-aT^p-bT^{p+1}) e^p_2(e_1+Te_2)^{-1}= \\ 
  & (cx^{p+1}+dx^py-ay^px-by^{p+1})e^{p-1}_2(e_2 \wedge e_1)^{-1}
  \in \H^0(\p^1_W,\mathcal{L}(p,-1)).
  \end{align*}
 After choosing a basis $\{e_1,e_2\}$ as above, we will denote by $\tilde{H}_i$ the local functions defining 
 $H_i$.

\end{notation}

We can also define automorphic vector bundles in a more systematic way. Let $G=\gsp_4$, $P$ be the Siegel
parabolic
$$
 P=\begin{pmatrix}
 * & * & * & * \\
 * & * & * & * \\
 0 & 0 & * & * \\
 0 & 0 & * & * 
 \end{pmatrix},
 $$
 and $B$ the upper triangular Borel. 
\begin{defn}
  Let $(\Lambda:=\Z^4_p=L \oplus L^\vee, \langle, \rangle)$ be the standard $4$-dimensional symplectic space over $\Z_p$, 
  with $L=\langle (1,0,0,0),(0,1,0,0) \rangle$ as Lagrangian. Further, let $F=\langle (1,0,0,0) \rangle \subset L$. 
  Consider 
  \begin{align*}
  %I_G&=\Isom_{\mathcal{O}_{\Sh}}(H^1_{\text{dR}}(A/\Sh),(L \oplus L^\vee) \otimes \mathcal{O}_{\Sh}) \\
  I_P&:=\Isom_{\mathcal{O}_{\Sh}}((\H^1_{\text{dR}}(A/\Sh),\omega),(L \oplus L^\vee, L) \otimes \mathcal{O}_{\Sh}) \\
  I_B&:=\Isom_{\mathcal{O}_{\flag}}((\H^1_{\text{dR}}(A/\flag),\omega,\LL),(L \oplus L^\vee, L,F) \otimes \mathcal{O}_{\flag}),
  \end{align*}
  where all the isomorphisms should respect the symplectic pairing up to a scalar. They are \'etale $P$ and $B$ torsors
  respectively, with the groups acting on the right. For any $\Z_p$-algebra $R$, define functors 
  \begin{align*}
  %F_G :& \Rep_{R}(G) \to \text{Coh}(\Sh_R) \\
   %& V \mapsto I_G \times^{G} V \\
   F_P :& \Rep_{R}(P) \to \text{Coh}(\Sh_R) \\
   & W \mapsto I_P \times^{P} W \\
   F_B :& \Rep_{R}(B) \to \text{Coh}(\flag_R) \\
   & V \mapsto I_B \times^{B} V 
  \end{align*}
  which produce automorphic vector bundles.
  \end{defn}

  For $W \in \Rep(P)$ we have $F_B(W)=\pi^* F_P(W)$. This follows 
  from the fact that $\pi^{-1}I_P$ is the pushout of $I_B$ along $B \subset P$. In this language 
  we can describe $\Omega^1_{\flag}$.
  \begin{lemma} \label{flag-differentials}
  There exists an isomorphism 
  $$
  e: T_{\flag/\Z_p} \cong F_{B}(\mathfrak{g}/\mathfrak{b}). 
  $$
  \begin{proof}
  Let $B_0 \subseteq \GL(\Lambda)$ be the upper triangular Borel so that $B_0 \cap \gsp_{4}=B$. The pushout of 
  $I_{B}$ along $B_0$ defines a functor $F_{B_0}$ and we have an embedding 
  $F_{B}(\mathfrak{g}/\mathfrak{b}) \hookrightarrow F_{B_0}(\mathfrak{gl}_{\Lambda}/\mathfrak{b}_0)$. 
  We can see the right-hand side as a compatible set of maps in 
  $\Hom_{\mathcal{O}_{\flag}}(\LL,H/\LL) \oplus \Hom_{\mathcal{O}_{\flag}}(\omega,H/\omega) \oplus 
  \Hom_{\mathcal{O}_{\flag}}(\LL^{\perp},H/\LL^{\perp})$. Similarly, we can see the
  left-hand side as the subspace of maps which are compatible 
  with the symplectic pairing, i.e. a compatible set of maps in $\Hom(\LL,H/\LL) \oplus \Hom(\omega,H/\omega)$.
  There is a map 
  $e': T_{\flag/\Z_p} \to F_{B_0}(\mathfrak{gl}_{\Lambda}/\mathfrak{b}_0)$ by sending $D$ to $\nabla_{D}$ on $H$. 
  Since $\nabla$ is compatible with the symplectic pairing it factors through $F_{B}(\mathfrak{g}/\mathfrak{b})$.
  It fits in the commutative diagram 
  $$
  \begin{tikzcd}
  0 \arrow[r]& T_{\flag/\Sh} \arrow[r] \arrow[d,"\sim"] & T_{\flag/\Z_p} \arrow[d,"e"] \arrow[r] &
   \pi^*T_{\Sh} \arrow[r] \arrow[d,"\text{ks}^{\vee}"] & 0 \\
  0 \arrow[r]& F_{B}(\mathfrak{p}/\mathfrak{b}) \arrow[r] & F_{B}(\mathfrak{g}/\mathfrak{b}) \arrow[r] & 
  F_{B}(\mathfrak{g}/\mathfrak{p}) \arrow[r] & 0 
  \end{tikzcd}
  $$
  since for $D \in T_{\flag/\Sh}$, we have that $\nabla_{D}(\omega) \subseteq \omega$ by functoriality of $\nabla$, and
  $F_{B}(\mathfrak{g}/\mathfrak{p})^{\vee}=\Sym^2 \omega$. The map $T_{\flag/\Sh} \to F_{B}(\mathfrak{p}/\mathfrak{b})=\LL(1,-1)$
  can be identified with the isomorphism in \Cref{basic-lemma-flag}(2), which implies that $e$ is an isomorphism. 
  \end{proof}
  \end{lemma}

\subsection{Weights and representation theory}

 We set up some notation for the weights and roots of $G=\gsp_4$. Fix $T \subset B \subset P \subset G$ where
 $T$ is the diagonal maximal torus, $B$ the upper triangular Borel, and $P$ the Siegel parabolic 
 with Levi quotient $M$.
 Weights $X^*(T)$ will be labelled by a triple 
 $(a,b,c)$ of integers such that $a+b=c \bmod{2}$, representing the character
 $$
 \begin{pmatrix}
 t_1 & & & \\
 & t_2 & & \\
 & & v t^{-1}_2 & \\
 & & & v t^{-1}_1
 \end{pmatrix}
 \to t^{a}_1 t^{b}_2 v^{(c-a-b)/2}.
 $$
 Then $c$ corresponds to the central character, from now on we will largely drop it from the notation
 and work with weights for $\text{Sp}_4$ instead. 
 \textbf{Warning}: with our convention on weights 
 we have $\LL(k,l)=F_B(l,k)$.
 \begin{itemize}
 \item Our choice of simple roots will be $\alpha=(1,-1,0)$ (short) and $\beta=(0,2,0)$ (long).
 Denote by $\rho=(2,1,3)$ the half sum of positive roots up to translation by the center.
 For a root $\gamma$ let 
 $x_{\gamma} \in \mathfrak{g}=\text{Lie}(G)$ be an element generating its weight space.
 \item After fixing $(B,T)$, the positive roots are $\Phi^{+}=\{\alpha, \beta, \alpha+\beta, 2\alpha+\beta\}$.
 The corresponding coroots are $\alpha^{\vee}=(1,-1)$, $\beta^\vee=(0,1)$, $(\alpha+\beta)^{\vee}=(1,1)$
 and $(2\alpha+\beta)^{\vee}=(1,0)$ with the obvious pairing.
 \item For a linear algebraic group $H/\Z_p$,
 let
 $\Rep_{\Z_p}(H)$ be the category of algebraic representations of $H$ which are finite free $\Z_p$-modules.
 \item  For a weight $\lambda \in X^*(T)$, 
 let $V(\lambda)\coloneqq \H^0(G/B^{-},\LL_{\lambda}) \in \Rep_{\Z_p}(G)$ be the dual Weyl module of $G$,
 where $B^{-}$ is the opposite Borel and $\LL_{\lambda}=(G \times \lambda)/B^{-}$.
 \item   
 The Weyl group of $G$ has a presentation $W=<s_0,s_1 : s^2_0,s^2_1,(s_0s_1)^4>$ with $s_0$ the reflection 
 corresponding to $\alpha$ and $s_1$ to $\beta$. The action on weights is 
 given by $s_0(a,b,c)=(b,a,c)$, $s_1(a,b,c)=(a,-b,c)$. Then $W_M=<s_0>$, and the set $W^M$ of minimal 
 length generators of $W_M \backslash W$ is given by $W^M=\{1,s_1,s_1 s_0,s_1s_0s_1\}$, ordered by increasing 
  length. 
 Define the dot action as $\omega \cdot \lambda=\omega(\lambda+\rho)-\rho$.
 \item We say that $\lambda \in X^*(T)$ is dominant, respectively $M$-dominant, if $\langle \lambda, \gamma^{\vee} \rangle \ge 0$
 for all simple roots $\gamma$ in $G$, respectively $M$. 
 \end{itemize}
 The $p$-restricted weights are $X_1(T)=\{ (a,b,c): 0 \le a-b, b < p\}$, and we say that 
 $\lambda \in X^*(T)$ is $p$-small if $|\langle \lambda+ \rho, \gamma^{\vee} \rangle| <p$ for 
 all roots $\gamma$. We divide $X^{*}(T)-\rho$ into $\rho$-shifted alcoves, defined as the 
 interior of the regions defined by the hyperplanes $H_{n,\gamma}=\{\langle \lambda+\rho, \gamma^{\vee} \rangle=np\}$
 for $\gamma \in \Phi$ and $n \in \Z$.
 There are four $p$-restricted alcoves
 defined by the weights $(a,b,c)-\rho$ satisfying
 \begin{align*}
 C_0&: a>b>0, a+b<p \\
 C_1&: a+b>p, b<a<p\\
 C_2&: a-b<p<a, a+b<2p \\
 C_3&: b<p, a+b>2p, a-b<p.
 \end{align*}

 \begin{itemize}
 \item For $\lambda$ dominant, let $L(\lambda)$ be the socle of $V(\lambda)$ as a $G_{\F_p}$-representation.
 They are always irreducible. 
 \item A Serre weight is an irreducible representation of $G(\F_p)$ with $\overline{\F}_p$ coefficients.
 Let $X_0(T)=\Z(0,0,2)$, then Serre weights are 
 in bijection with $F(\lambda)\coloneqq L(\lambda)_{\mid G(\F_p)}$ for $\lambda=(a,b) \in X_1(T)/(p-1)X_0(T)$.
 In practice we will drop the central character. 

 \end{itemize}

 For weights in one of the four alcoves we can describe how $V(\lambda)_{\Fpbar}$ decomposes into irreducibles. 
 \begin{prop} \cite[\S 7]{Jantzen-decomp}\label{Weyl-decomposition}
 Let $\lambda=(a,b,c) \in X_1(T)$. We have the following exact sequences in $\Rep_{\overline{\F}_p}(G)$
 \begin{align*}
   0 \to L(\lambda) \to  &V(\lambda) \to L(2p-b-3,2p-a-3,c) \to 0 \text{ if } \lambda \in C_3  \\
   0 \to L(\lambda) \to  &V(\lambda) \to L(2p-a-4,b,c) \to 0 \;\;\;\text{ if } \lambda \in C_2 \\
   0 \to L(\lambda) \to &V(\lambda) \to L(p-b-3,p-a-3,c) \to 0  \text{ if } \lambda \in C_1 \\
   &V(\lambda)=L(\lambda) \text{ if } \lambda \in C_0 
  \end{align*}
  If $\lambda \in X_1(T) \setminus \cup^3_{0} C_i$ then $V(\lambda)$ is irreducible unless $a-b=p-1$ and 
  $p/2 <b+1<p$, or $p=2$ and $\lambda=(1,1,c)$. In this range all the $L(\lambda)$ are still 
  irreducible as $G(\F_p)$-representations.
 \end{prop}

\subsection{Hecke operators away from $p$}
Fix a level $K=K^pK_p \subset G(\mathbb{A}^{\infty}_{\Q})$, and let $\mathbb{A}^{\infty,p}$
be the finite adeles away from $p$. Consider $
\mathcal{H}^{p}=C^{\infty}_{c}(G(\mathbb{A}^{\infty,p})//K^p,\Z_p),
$ 
the abstract Hecke algebra away from $p$
defined as a convolution algebra of locally constant, compactly supported, 
bi-$K^p$ invariant functions on $G(\mathbb{A}^{\infty,p})$ with coefficients in $\Z_p$.
We define the action of $\mathcal{H}^{p}$ on coherent cohomology via Hecke correspondences.
Let $g \in G(\mathbb{A}^{p,\infty})$, and $K_g=K \cap gKg^{-1}$. There are two finite \'etale maps over $\Z_p$ 
$p_1,p_2: \text{Sh}_{K_g} \to \text{Sh}_{K}$ defined as follows. The first one is defined on points of 
\Cref{Shimura-variety} by 
$(A,\lambda,\eta_{K_g}) \to (A,\lambda,\eta_{K})$ where $\eta_K$ is obtained by taking any $K$-orbit of any 
$\eta$ in $\eta_{K_g}$. The second  map is defined by composing $\text{Sh}_{g^{-1}K_g g} \to \text{Sh}_{K}$ defined as 
$p_1$ with an isomorphism $[g]: \text{Sh}_{K_g} \to \text{Sh}_{g^{-1}K_g g}$ sending $A$ to the unique 
$A'$ prime to $p$ quasi-isogenous to $A$ via $f: A \to A'$ satisfying the conditions in
\cite[Prop 1.4.3.4]{Lan-thesis}. This correspondence only depends on the coset $K^pgK^p$, so in particular
one can check that 
it induces 
an action of $\mathcal{H}^{p}$ by correspondences. Moreover, for $Q \in \{B,P\}$, we have natural 
isomorphisms of torsors $T_{g}: p^*_1 I_{Q} \cong p^*_2 I_{Q}$. This induces an action in cohomology: for $V \in \Rep_{\Z_p}(P)$
$$
T_{g}: \H^i(\text{Sh}_K, F_{P}(V)) \xrightarrow{p^*_2} \H^i(\text{Sh}_{K_g}, p^*_2 F_{P}(V))
\xrightarrow{T_{g}} \H^i(\text{Sh}_{K_g}, p^*_1 F_{P}(V)) \xrightarrow{\text{Tr} p_1 } \H^i(\text{Sh}_K, F_{P}(V)),
$$
and analogously for $B$.
We say that a collection of maps $F_{Q}(V)_{K} \to F_{Q}(W)_{K}$
of automorphic vector bundles, ranging across $K=K^pK_p$ neat with $K_p$ hyperspecial 
is Hecke equivariant away from $p$ if the associated maps on cohomology are equivariant for $\mathcal{H}^{p}$.
In practice, it is enough to check that the maps are compatible with base change and 
prime to $p$ quasi-isogenies $f: A \to A'$, in the 
sense that they are compatible under replacing the torsors $I_{Q}$ with the ones defined by $H^1_{\dR}(A'/\text{Sh}_K)$.

We explain this in the case of the Kodaira--Spencer isomorphism, which is the only non-trivial step occurring 
in the 
case of theta operators. 
Let $\delta\coloneqq \omega(0,0,2)$.
It is a line bundle which makes the symplectic pairing
 $H \otimes H \to \delta$ Hecke equivariant. It is non-canonically trivial via the quasi-polarization,
and $\delta^{p-1}=\mathcal{O}_{\Shbar}$ canonically.
Then we can write 
$$
\LL(a,b,c)\coloneqq F_{B}(b,a,c)=\LL(a,b)\otimes \delta^{\frac{c-a-b}{2}}, \;\;\;
\omega(a,b,c)\coloneqq F_{P}(W(a,b,c))=\omega(a,b) \otimes \delta^{\frac{c-a-b}{2}},
$$
so that the central character does not change the automorphic vector bundle, 
it will only keep track of the Hecke action. Using the quasi-polarization we can write the Kodaira--Spencer isomorphism as a map 
$\text{ks}_{A,\lambda} : \omega(2,0,0)=\Sym^2 \omega_{A} \otimes \delta^{-1}_{A} \cong \Omega^1_{\Sh/\Z_p}$. 
Given a prime to $p$ quasi-isogeny between polarized abelian varieties $f: (A,\lambda) \to (A',\lambda')$ 
over $\Sh$ which preserves $\lambda$ and $\lambda'$ up to scalars, the compatibility needed to show 
that $\text{ks}_{A,\lambda}$ is Hecke equivariant is the commutativity 
of the square (see \cite[Lem 4.1.2]{Eischen-Mantovan-1})
$$
\begin{tikzcd}
\Sym^2 \omega_A \otimes \delta^{-1}_A \arrow[r,"\text{ks}_{A,\lambda}"] & \Omega^1_{\Sh/\Z_p} \arrow[d,equal] \\
\Sym^2 \omega_{A'} \otimes \delta^{-1}_{A'}  \arrow[u,"df"] \arrow[r,"\text{ks}_{A',\lambda'}"] & \Omega^1_{\Sh/\Z_p}.
\end{tikzcd}
$$
Similarly, the isomorphism in \Cref{flag-differentials} is Hecke equivariant away from $p$.

Given a level $K$ let $S$ be a finite set of places containing $p$ 
such 
that $G(\Z_{\ell}) \subseteq K$ for $\ell \notin S$.
 Let $\mathbb{T}/\Z_p$ be the spherical Hecke algebra 
away from $S$, defined as a restricted tensor product of local Hecke algebras 
$\bigotimes^{'}_{\ell \notin S} C^{\infty}_{c}(G(\Q_l)//K_l,\Z_p)$.
It is a commutative subalgebra of $\mathcal{H}^{p}$, so it also acts on coherent cohomology. 
Although it depends on the level $K$ we will suppress it from the notation. 
Let $\mathfrak{m} \subset \mathbb{T}$ be a maximal ideal occurring in the coherent cohomology of
the toroidal compactification $\Shbar^{\tor}$, e.g. occurs in $\H^0(\Shbar,\omega(k,l))$
for $k \ge l \ge 0$.
Then one can attach a semisimple Galois representation $\overline{r}_{\mathfrak{m}}: G_{\Q} \to \gsp_4(\overline{\F}_p)$
 by the construction of Galois representations for automorphic representations of regular weight \cite{Galois-reps},
 and the use of generalized Hasse invariants in \cite{boxer-thesis} or \cite{Goldring-Koskivirta-Galois}. 
We say that $\mathfrak{m}$ is non-Eisenstein if $\overline{r}_{\mathfrak{m}}$ is irreducible 
as a $\GL_4$-valued representation.

\subsection{Vanishing results for coherent cohomology}
The Siegel threefold is not proper, so we introduce the standard machinery of toroidal compactifications. 

\begin{prop} \cite[\S 2.3]{Lan-Polo}
  Let $C$ be the cone of positive semidefinite symmetric bilinear forms on $\mathbb{R}^2$
  whose radicals are defined over $\Q$, and $\Sigma$ a smooth projective $K$-admissible decomposition of $C$ into polyhedral cones
  as in \cite[Def 6.3.3.4 and 7.3.1.3]{Lan-thesis}.
  By \cite[Thm 6.4.1.1 and 7.3.3.4]{Lan-thesis} there exists an associated toroidal compactification
   $\Sh^{\tor, \Sigma}$ over $\Z_p$. It is a smooth 
  proper scheme satisfying the following properties. 
  \begin{enumerate} \label{toroidal}
  \item The boundary $D\coloneqq \Sh^{\tor,\sigma}-\Sh$ with its reduced structure 
  is a Cartier divisor with simple normal crossings. 
  \item The universal abelian scheme extends to a semi-abelian scheme $\pi^{\tor}: A^{\tor} \to \Sh^{\tor}$.
  The prime to $p$ polarization extends to a prime to $p$ isogeny 
  $\lambda: A^{\tor} \to A^{\tor,\vee}$. Define $\omega^{\tor}=e^* \Omega^1_{A^{\tor}/\Sh^{\tor}}$, it extends $\omega/\Sh$.
  \item There is a canonical extension $H^{\tor}$ of $H$ as a symplectic vector bundle over $\Sh^{\tor}$.
  It fits in the exact sequence
  $$
  0 \to \omega^{\tor} \to H^{\tor} \to \omega^{\tor, \vee}_{A^{\vee}} \to 0 
  $$
  which makes $\omega^{\tor}$ into a $P$-torsor. 
  \item The functor defining automorphic vector bundles 
extend to $F_{P}^{\can}: \Rep_{R}(P) \to \textnormal{Coh}(\Sh^{\tor}_R)$ using the torsors given 
by $\omega^{\tor} \subset H^{\tor}$. We define subcanonical extensions by $F^{\sub}_{P}(V)=F^{\can}_P(V)(-D)
$. This defines $\omega^{?}(k,l)$ for $? \in \{\can, \sub\}$. 
From now on we will largely drop the indices for canonical 
extensions. 
 \item The coherent cohomology groups $\H^i(\Sh^{\tor,\Sigma},\omega^{?}(k,l))$ for $? \in \{\can, \sub\}$
 are independent of the cone decomposition $\Sigma$. 
 \item We have $\H^0(\Sh^{\tor}_{R},\omega(k,l))=\H^0(\Sh_{R},\omega(k,l))$ for $R \in \{\Z_p,\F_p\}$, by 
 \cite{Higher-Koecher}. 
  \end{enumerate}
  \end{prop}

  With that we can recall some important vanishing theorems for mod $p$ coherent cohomology.
  \begin{theorem} \label{lan-suh}
 \begin{enumerate}
 \item Let $p \ge 5$. For $k \ge l \ge 5$ satisfying $(p-1)(k-2)>(p+1)(l-4)$
 and $p(l-4)>k-2$ (in particular this holds if
 $l \le p+2$, $k-l<p$ and $l \ge 6$) we have  
 $\H^i(\Shbar^{\tor},\omega^{\sub}(k,l))=0$ for $i \ge 1$.
 For $k \ge l \ge 4$ with 
 $k-l<p-3$, we have $\H^i(\Shbar^{\tor},\omega^{\sub}(k,l))=0$ for $i \ge 1$.
 \item If $\H^0(\Shbar,\omega(k,l))\neq 0$, then $k \ge l$ and  $k+lp\ge 0$ 
 \cite[Thm 5.1.1]{cone-conjecture-gsp4}.

 \end{enumerate}
 \begin{proof}
 The first case of $(1)$ follows by the results on the ample cone on $\flag$ in 
 \cite[Ex 4.30, Thm 5.10]{alexandre},
 and the second case by \cite[Thm 8.13]{Lan-Suh-non-compact}. 
 \end{proof}
  \end{theorem}

 \subsection{Deformation theory}
  It will sometimes be useful to make computations on a formal neighbourhood of a point.
  Let $q \in \Sh(\overline{\F}_p)$ be a geometric point. Let $(A_q,\lambda,\eta)$
  be one representative of $q$, and denote by $\text{Def}(A_q,\lambda,\eta)$ be the moduli space of 
  deformations of $q$
  to local Artin algebras with an augmentation $R \twoheadrightarrow \overline{\F}_p$. Then 
  $$
 \text{Spf } \widehat{\mathcal{O}}_{\Sh,q} \cong \text{Def}(A_q,\lambda,\eta)
  $$
  as functors on Artin algebras, using the moduli interpretation of $\Sh$ in \Cref{Sh-moduli-isomorphisms}.
  A classical result of Grothendieck then says that 
  $
  \widehat{\mathcal{O}}_{\Sh,q} \cong W(\overline{\F}_p)[[T_{11},T_{12},T_{22}]]
  $
  and similarly, if $r \in \flag$ is a geometric point 
  $
  \widehat{\mathcal{O}}_{\flag,r} \cong W(\overline{\F}_p)[[T_{11},T_{12},T_{22},T]]
  $
  since $\flag$ is locally a $\p^1$ over $\Sh$.

 Let $(R,I)$ be the divided power envelope of $(\mathcal{O}_{\Sh_{\Z_p}},\mathfrak{m}_{q})$, and let $I^{[n]}$ be the ideal of $R$
 generated by elements $\prod \gamma_{n_i}(x_i)$ for $x_i \in \mathfrak{m}_{q}$ and $\sum n_i \ge n$, so that 
 $S_n\coloneqq R/I^{[n]}$ is a divided power thickening of $\kappa(q)$. Grothendieck--Messing theory 
 identifies $\H^1_{dR}(A/S_n)$ with the Dieudonne crystal  $D(A[p^\infty]_{S_n})$
 of the p-divisible group $A[p^\infty]_{S_n}$. In particular 
 $H_{S_n}$ is equipped with a natural map $\nabla: H_{S_n} \to H_{S_n} \otimes \Omega^1_{S_n/\Z_p,\delta}$
  in the sense of \cite[\href{https://stacks.math.columbia.edu/tag/07J6}{Tag 07J6}]{stacks-project}
 by virtue of being a crystal, and $\nabla$
 is identified with the pullback of the Gauss-Manin connection to $S_n$. 
 There is a map of PD thickenings $(W(\overline{\F}_p),\overline{\F}_p) \to 
(S_n \to \kappa(q))$ corresponding to the natural section of 
$S_n/p \to S_n/I \cong \kappa(q)$. Since the Dieudonne module is a crystal
\begin{equation} \label{blabla}
 \H^1_{dR}(A/S_n) \cong D(A[p^\infty]_q)(S_n \to \kappa(q)) \cong D(A[p^\infty]_q)(W(\overline{\F}_p),\overline{\F}_p)
  \otimes_{W} S_n
\end{equation}
\begin{lemma} \label{gm-gauss-manin} \cite[\href{https://stacks.math.columbia.edu/tag/07J6}{Tag 07J6}]{stacks-project}
Let $\{e_i\}$ be a basis of $D(A[p^\infty]_q)(W(\overline{\F}_p),\overline{\F}_p)$. Then
under the isomorphism \eqref{blabla} the natural map 
$\nabla: H_{S_n} \to H_{S_n} \otimes \Omega^1_{S_n/\Z_p,\delta}$ satisfies $\nabla(e_i)=0$.
\end{lemma}

   On the formal completion of an ordinary point one has the richer structure of Serre--Tate coordinates.
   The following theorem is a compatibility 
   result between the Serre--Tate and Grothendieck--Messing coordinates.

  \begin{theorem} \cite{Katz-Serre-Tate} \label{ST-coordinates}
  Let $s \in \Shbar^{\textnormal{ord}}(\overline{\F}_p)$. Then there exists a basis $\{e_1,e_2\}$ of 
  $\omega_{\widehat{\mathcal{O}}_{\Shbar,s}}$ such that $V e_i=e^{(p)}_i$ and 
  $D_{ij}\coloneqq \textnormal{ks}(e_i e_j)^{\vee}=(T_{ij}+1) \frac{\partial}{\partial T_{ij}}$.
  \end{theorem}

  A computation on the formal neighbourhood of various points 
yields the following result.   \begin{prop} \label{Hasse-simple-zero}
    Let $i=1,2$. 
    Then $H_i$ vanishes with a simple zero on $D_i$.   
    \begin{proof}
    The Zariski-Nagata theorem for regular algebras over a perfect field 
    allows us check the first statement on the formal completion of any
    closed point $q \in D_i$. For $H_1$ we choose $\tilde{q} \in D_1(\overline{\F}_p)$
    mapping to $q \in \Shbar^{=1}(\overline{\F}_p)$. Write  $\widehat{O}_{\Shbar,q}=\overline{\F}_p[[T_{11},T_{12},T_{22}]]$. 
    Let $R=\mathcal{O}_{\Shbar,q}/ \mathfrak{m}^{p}$, which is a PD thickening
    of $\kappa(q)$. Grothendieck--Messing theory tells us that $\H^1_{\text{dR}}(A/R)=\H^1_{\text{dR}}(A/\kappa(q)) \otimes R$.
    By the classification of Ekedahl-Oort strata we have that $A[p]_{\kappa(q)}$ is isomorphic
    to the product of the $p$-torsion of an ordinary elliptic curve
    and the $p$-torsion of a supersingular elliptic curve. We can therefore
    choose a symplectic basis $\{\overline{e}_i\}$ of $\H^1_{\kappa(q)}$ so that $V_{\kappa(q)}$  has matrix
    \begin{equation} \label{V-p-rank-1}
    V_{\kappa(q)}=\begin{pmatrix}
    1 & 0 & 0 & 0 \\
    0 & 0 & 1 & 0 \\
    0 & 0 & 0 & 0 \\
    0 & 0 & 0 & 0
    \end{pmatrix}
    \end{equation}
    with respect to the basis $\{\overline{e}_i\}$ and $\{\overline{e}^{(p)}_i\}$. 
    By Grothendieck--Messing theory lifts of $\omega_{\kappa(q)}$ to $H^1_{R}$ parametrize deformations of $A/\kappa(q)$. Since $A/R$
    is the universal deformation we can assume without loss of generality that
    \begin{equation} \label{omega_R}
    \omega_R=(\overline{e}_1+T_{12}\overline{e}_3+T_{11}\overline{e}_4,\overline{e}_2+T_{22}\overline{e}_3
    +T_{12}\overline{e}_4) \subset H^1_{R},
    \end{equation}
    On the other hand $A^{(p)}/R\coloneqq A_R \times_{R,\text{F}_R} R$ is the trivial deformation of $A/\kappa(q)$,
    since Frobenius factors through $\kappa(q)$. Therefore,
    $
    \omega^{(p)}_R=(\overline{e}^{(p)}_1,\overline{e}^{(p)}_2).
    $
    We can compute the matrix of $V_R: \omega_R \to \omega^{(p)}_R$
    with respect to the basis described above: 
    \begin{equation} 
    V_R=\begin{pmatrix}
    1 & 0 \\
    T_{12} & T_{22}
    \end{pmatrix}.
    \end{equation}
    Therefore in this basis $\tilde{H}_1=T_{22}$, which vanishes with a simple zero. 
    For $H_2$ we choose $\tilde{q} \in D_2(\overline{\F}_q)$
    mapping to $q \in \Shbar^{\text{ord}}(\overline{\F}_p)$. In the basis of \Cref{ST-coordinates}
     on the formal completion 
    of $q$ we can write $\tilde{H}_2=T-T^p$, so that
     the maximal ideal of $\widehat{\mathcal{O}}_{\flag,\tilde{q}}$ is $(T_{11},T_{12},T_{22},T-\alpha)$
     for some $\alpha \in \F_p$, and it is clear that $H_2$ has a simple zero. 
   
    \end{proof}
    \end{prop}

 \section{Construction of the theta operators on the open strata} \label{Section2}

Let $U=U_1 \cap U_2$ be the open in $\flag$ where $H_1 \cdot H_2$ does not vanish.
We first construct the basic theta operators over $U$, and then we extend them to $U_1$.

\subsection{The Igusa variety}
We describe an Igusa variety $\Ig/\Shbar^{\text{ord}}$, which is a finite \'etale cover.  
Over $\Shbar^{\text{ord}}$ there is a connected-\'etale sequence for $A[p]$ consisting of finite flat subgroups
$$
0 \to A[p]^0 \to A[p] \to A[p]^{\text{\'et}} \to 0.
$$
\begin{defn}
Let $\tau: \Ig \to  \Shbar^{\text{ord}}$ be the scheme representing the moduli problem sending $S \in \text{Sch}/\F_p$
to 
$$
\Ig(S)=\{(A,\lambda, \eta) \in \Shbar^{\text{ord}}(S) \text{ and } \phi: (\Z/p)^2 \cong A[p]^{\text{\'et}}\}.
$$
It is an \'etale $\GL_2(\F_p)$-torsor over $\Shbar^{\text{ord}}$, with $\GL_2(\F_p)$ acting on $\phi$
via $(\Z/p)^2$.
\end{defn}
Taking Cartier duals, one also gets the isomorphism $\mu^2_p \cong A[p]^0$ over $\Ig$,
and in fact $\Ig$ also parametrizes embeddings $\phi: \mu^2_p \hookrightarrow A[p]$. This implies that we can trivialize 
$\omega$ over $\Ig$:
$$
\tau^* \omega=\omega_{A[p]_\Ig}=\omega_{A[p]^{0}_\Ig}=\omega_{\mu^2_p} \cong \mathcal{O}^2_{\Ig}.
$$
Since we can explicitly describe $V$ on $\mu_p$ it follows that there is a basis $\{e_1,e_2\}$ of $\omega_\Ig$
satisfying $V e_i=e^{(p)}_i$. We will refer to it as a canonical basis of $\omega_{\Ig}$, keeping in mind 
that it is only canonical up to the action of $\GL_2(\F_p)$.

\begin{notation} \label{ig-coordinates}
In order to make concrete computations we give some local coordinates for $\Ig$. Let $W \subseteq \Shbar$
be an open set and $\{e_1,e_2\}$ a basis of $\omega_{W}$. Over $(W \cap \Shbar^{\text{ord}})_{\Ig}$
there is some matrix $M=(x_{ij}) \in \GL_2(\mathcal{O}_{\Ig})$ such that $M\begin{pmatrix}
  e_1 \\
  e_2
\end{pmatrix}=
\begin{pmatrix}
  \tilde{e}_1 \\
  \tilde{e}_2
\end{pmatrix}$ where $\{\tilde{e}_1,\tilde{e}_2\}$ is the canonical basis of $\omega_{\Ig}$. Since 
$V \tilde{e}_i=\tilde{e}^{(p)}_i$ we have the relation 
$M^{(p)}=M V^{\text{T}}$, where $M^{(p)}=(x^p_{ij})$ and $V$ is the matrix of $V: \omega \to \omega^{(p)}$ with respect to 
the basis $\{e_1,e_2\}$. 
\end{notation}

Recall that $U_1=\pi^{-1}(\Shbar^\text{ord})$, so that $\flag_{\Ig}\coloneqq U_1 \times_{\Shbar^{\text{ord}}}
\Ig=\mathbb{P}(\omega_\Ig) \cong \Ig \times
\mathbb{P}^1$ is a $\GL_2(\F_p)$-\'etale torsor over $U_1$, with $\GL_2(\F_p)$ acting diagonally on $\Ig \times
\mathbb{P}^1$. 

\begin{prop} \label{igusa-splitting}
The exact sequence 
$$
0 \to \pi^* \Omega^1_{\Shbar} \to \Omega^1_{\flag} \to \Omega^1_{\flag/\Shbar} \to 0 
$$
canonically splits over $U_1$. We denote by $s : \Omega^1_{U_1/\overline{\F}_p} \to \pi^* \Omega^1_{\Shbar}$ the associated section. 
\begin{proof}
By Galois descent to construct $s$ it is enough 
to construct a $\GL_2(\F_p)$-equivariant section $\Omega^1_{\flag_\Ig} \to \pi^* \Omega^{1}_\Ig$.
Both projections from $\flag_{\Ig} \cong \Ig \times \mathbb{P}^1$ are $\GL_2(\F_p)$-equivariant,
 so the splitting of the exact sequence for $\Omega^1_{\Ig \times \mathbb{P}^1}$
is also $\GL_2(\F_p)$-equivariant. 
\end{proof}
\end{prop}

There is also a more group-theoretic way to understand this section. Over $U_1$ there is an 
$M \cap B$-reduction $I_{M \cap B} \subseteq I_{B}$, defined as the elements of $I_{B}$
mapping the direct sum $H=\omega \oplus \Ker(V)$ to $(L \oplus L^{\vee}) \otimes \mathcal{O}_{U_1}$.
This defines a functor $F_{M \cap B}: \Rep(M \cap B) \to \text{Coh}(U_1)$.
\begin{prop} \label{splitting-2}
Let $t: \mathfrak{g}/\mathfrak{p} \to \mathfrak{g}/\mathfrak{b}$ be given by the inclusion of 
the $x_{-\gamma}$ for $\gamma \in \{\beta,\alpha+\beta,2\alpha+\beta\}$.
It is a map of $M \cap B$-representations, and $F_{M \cap B}(t^{\vee})$
over $U_1$ is identified with the splitting in 
\Cref{igusa-splitting} under the isomorphism from \Cref{flag-differentials}.
\begin{proof}
This can be checked over $U_{1,\Ig}$, where there is a canonical trivialization of $H=\omega \oplus \Ker(V)$. 
Then $I_{M \cap B}$ is identified with the canonical $M \cap B$ torsor on $\mathbb{P}^1=M/(M \cap B)$, and 
we can easily identify $F_{M \cap B}(t^{\vee})$ with the splitting for $\Omega^1_{\mathbb{P}^1 \times \Ig}$.
\end{proof}
\end{prop}

\subsection{Construction on the open strata $U$}

On $U \subset \flag_{\overline{\F}_p}$ we have both the splittings
 $H=\mathcal{L} \oplus V^{-1}(\mathcal{L}^{(p)})$ and $H=\omega \oplus \omega^{\vee}$, so that 
 setting $\mathcal{L}'\coloneqq \omega \cap V^{-1}(\LL^{(p)})$ we get a full splitting of the symplectic flag 
 on $U$ into line bundles: 
 $$
 H=\LL \oplus \LL' \oplus (\LL^{'\perp} \cap \omega^\vee) \oplus (\LL^{\perp} \cap \omega^\vee).
 $$
Denote by $\pi_{\LL}: H \to \LL$ and $\pi_{\LL'}: H \to \LL'$ the corresponding projections,
we will use the same notation for the projections restricted to $\omega$. 
Moreover, we have $\LL(k,l)_{U}=\LL'^{k} \otimes \LL^l$.

\begin{defn}(Theta operators on $U$) \label{operators-on-U}
Let $\lambda=(k,l) \in X^*(T)$. We define the following maps over $U$.
%\textcolor{red}{$\pi_2$ should in fact be
%2$\pi_2$ by considering the splitting of $\Sym^2 \omega$, or alternative the map $\Sym^2 \omega \to \omega^{\otimes 2}$
%given by $xy \mapsto x\otimes y + y \otimes x$.}
\begin{itemize}
\item The splitting $\omega=\LL \oplus \LL'$ induces a splitting of 
$\Sym^2 \omega$. The corresponding projections are
\begin{align*}
\pi_1 :&  \pi^* \Omega^1_{\Shbar} \cong \Sym^2 \omega \to \LL'^2=\LL(2,0)  \\
\pi_2 :& \pi^* \Omega^1_{\Shbar} \cong \Sym^2 \omega  \to \LL' \otimes \LL=\LL(1,1) \\
\pi_3 :& \pi^* \Omega^1_{\Shbar} \cong \Sym^2 \omega  \to  \LL^2=\LL(0,2).
\end{align*} 
\item For $i=1,2,3$ we define operators $\tilde{\theta}_i: \LL(\lambda) \to \LL(\lambda+\lambda_i)$
with $\lambda_1=(2,0)$, $\lambda_2=(1,1)$, $\lambda_3=(0,2)$. Let $L$ be $\LL$ or $\LL'$. Then for 
$\LL(\lambda)=L$ the operators $\tilde{\theta}_i$ are defined as
$$
\tilde{\theta}_i: L \hookrightarrow H \xrightarrow{\nabla}  H \otimes \Omega^1_{\flag/\overline{\F}_p}
\xrightarrow{\pi_{L} \otimes s} L \otimes \pi^* \Omega^1_{\Shbar} \xrightarrow{\text{id} \otimes \pi_i}
L \otimes \LL(\lambda_i).
$$
For a general weight $\lambda$, it is defined by prescribing that $\tilde{\theta}_i(fg)=f \tilde{\theta}_i(g)
+\tilde{\theta}_i(f)g$ for local sections $f,g$, and that on $\mathcal{O}_{\flag}$ it is given by composing 
the differential with $\pi_i \circ s$. Explicitly: on a local section $fe^{-1} \in L^{-1}$ with $f \in \mathcal{O}$
and $e \in L$ a local basis
$$
\tilde{\theta}_i(fe^{-1}) \coloneqq \pi_i\circ s(df) \otimes e^{-1}-f\tilde{\theta}_i(e) \otimes e^{-2}.
$$
Every $\LL(k,l)$ is a tensor product of 
$\LL, \LL'$ or their duals, so that $\tilde{\theta}_i$ can be defined by the Leibniz rule.
\item Let $r: \Omega^1_{\flag} \to \Omega^1_{\flag/\Shbar}$, and $L$ as before.
Define $\tilde{\theta}_4$ on $L$ as
$$
\tilde{\theta}_4: L \hookrightarrow H \xrightarrow{\nabla}  H \otimes \Omega^1_{\flag}
\xrightarrow{\pi_{L} \otimes r} L \otimes \Omega^1_{\flag/\Shbar}=L \otimes \LL(-1,1),
$$
using \Cref{basic-lemma-flag}(2). We extend it to arbitrary weights as above. 
\end{itemize}

\end{defn}

\begin{remark}
In the definition of $\tilde{\theta}_i$ it is natural to use the projection $\pi_{L}$
as opposed to the projection to any other line bundle, in which case 
one would get a linear map (in some cases identically zero).
See also \Cref{theta-on-X} ahead for the relation of the extension 
of $\tilde{\theta}_i$ to $\flag$ to the way \cite{Eischen-Mantovan-1} define theta operators 
on the Shimura variety. 
\end{remark}

We also remark that $\pi_{1,2,3}$ can be described via the $T$-reduction $I_{T} \subset I_{B}$ over $U$ 
given by the splitting
of the Hodge filtration, as in \Cref{splitting-2}.
Since $\theta_{1,2,3,4}$ are maps of line bundles on $U$ it is 
a formal consequence that they can be extended to $\flag$ by multiplying by some sufficiently high 
power of the Hasse invariants \footnote{One can see this by considering degree $1$ differential operators as linear maps
$P^1_{\Shbar} \otimes V \to W$.}. Moreover, it suffices to 
extend 
the theta operators for weights $(1,0)$, $(0,1)$ and $(0,0)$, since all other weights are a combination of tensor products
and duals of these.

\subsection{Extension to $U_1$}
Let $L=\LL(1,0)$ or $\LL(0,1)$.
There is a finite \'etale cover  $\tau: U_{\Ig} \to U$ by pulling back via the map $\Ig \to \Shbar^\text{ord}$,
and one can make sense
of $\tau^* \tilde{\theta}_i$, defined by the composition of the same maps pulled-back to $U_{\Ig}$,
since the pullback of the Gauss-Manin connection under an \'etale map
is again the Gauss-Manin connection. Thus, we can prove properties about $\tilde{\theta}_i$ on 
$U_{\Ig}$, where there is a basis $\{e_1,e_2\}$ of $\omega_{\Ig}$ satisfying $V e_i=e^{(p)}_i$.

\begin{lemma} \label{magic-on-U1}
Consider the map over $U_1$
$$
\tilde{\theta}: \omega \hookrightarrow H \xrightarrow{\nabla} H \otimes \Omega^1_{\flag} \xrightarrow{\pi_\omega \otimes s}
\omega \otimes \pi^* \Omega^1_{\Shbar},
$$
where $\pi_\omega: H \to \omega$ is the unit root splitting over $U_1$. 
Then $\tilde{\theta}$ satisfies $\tilde{\theta}(\LL) \subset \LL \otimes \pi^* \Omega^1_{\Shbar}$, so that it induces maps 
$$
\theta_L : L \to L \otimes \pi^* \Omega^1_{\Shbar}.
$$
\begin{proof}
It is enough to prove it after pulling back to $\Ig$. The analogous map $\omega \to \omega \otimes \pi^*\Omega^1_{\Ig}$
has $\{e_1,e_2\}$ as a horizontal basis, since the kernel of $\pi_\omega$ is $\Ker(V: H \to \omega^{(p)})$
and
\[
V \nabla(e_i)=\nabla^{(p)}(Ve_i)=\nabla^{(p)}(e^{(p)}_i)=0
\]
by functoriality of $\nabla$, and the fact that the differential of Frobenius vanishes. Therefore, in this basis
$\tilde{\theta}$ is given by the composition of the trivial connection on $\mathcal{O}^2_{\Ig \times \p^1}$
and the section $s$, so it ignores any differentiation along $\p^1$.  It follows that we can choose a local basis 
for $L$ that is horizontal for $\tilde{\theta}$.
\end{proof}
\end{lemma}

To extend $\tilde{\theta}_{1,2,3}$ to $U_1$, 
we use the maps 
$\theta_{L}: L \to L \otimes \pi^* \Omega^1_{\Shbar}$ from \Cref{magic-on-U1}. They extend the analogous maps 
in the
definition of $\tilde{\theta}_{1,2,3}$. Then it suffices to extend $\pi_{1,2,3}$ to $U_1$. 
The map $\pi_{(1,0)}: \omega \to \omega/\LL$
extends $\pi_{\LL'}: \omega \to \LL'$ to $U_1$, and the map
\[
\pi_{(0,1)}: \omega \xrightarrow{V} \omega^{(p)} \to (\omega/\LL)^{(p)}=\LL(p,0)
\]
extends $\pi_{\LL} : \omega \to \LL$, in the sense that $\pi_{(0,1) \mid U}=H_2 \cdot \pi_{\LL}$.
Consider the map $\Sym^2 \omega \to \omega^{\otimes 2}$ given by 
$x y \mapsto x \otimes y + y \otimes x$. Then
\begin{align} \label{projections-from-KS}
  \Pi_1 :&  \Sym^2 \omega \xrightarrow{\Sym^2 \pi_{(1,0)}} \LL(2,0)  \\
  \Pi_2 :& \Sym^2 \omega \to \omega^{\otimes 2} \xrightarrow{\pi_{(1,0)} \otimes \pi_{(0,1)}} \LL(p+1,0) \\
  \Pi_3 :& \Sym^2 \omega  \xrightarrow{\Sym^2 \pi_{(0,1)} } \LL(2p,0)
\end{align}
extend the $\pi_i$ to $U_1$, and we have $\Pi_i=H^{i-1}_2 \pi_i$ on $U$. 
\begin{defn}(Extending $\tilde{\theta}_{1,2,3}$ to $U_1$)
Let $\lambda \in X^*(T)$, define the maps of sheaves on $U_1$ 
\begin{align*}
\Theta_1: &\LL(\lambda) \to \LL(\lambda+(2,0))\\
 \Theta_2: &\LL(\lambda) \to \LL(\lambda+(p+1,0))\\
 \Theta_3: &\LL(\lambda) \to \LL(\lambda+(2p,0))
\end{align*}
as follows.
For $\lambda=(1,0)$ or $(0,1)$, $L=\LL(\lambda)$  they are defined as
\begin{align*}
 \Theta_1: & L \xrightarrow{\theta_L} L \otimes \pi^* \Omega^1_{\Shbar} \xrightarrow{\text{id} \otimes \Pi_1}
 L \otimes \LL(2,0) \\
 \Theta_2: & L \xrightarrow{\theta_L} L \otimes \pi^* \Omega^1_{\Shbar} \xrightarrow{\text{id} \otimes \Pi_2}
 L \otimes \LL(p+1,0) \\
 \Theta_3: & L \xrightarrow{\theta_L} L \otimes \pi^* \Omega^1_{\Shbar} \xrightarrow{\text{id} \otimes \Pi_3}
 L \otimes \LL(2p,0), \\
\end{align*}
and for $\lambda=(0,0)$ in the same way using $\mathcal{O}_{U_1} \xrightarrow{s \circ d} \pi^*\Omega^1_{\Shbar}$
instead of $\theta_{L}$.
They are extended to arbitrary $\lambda$ as in \Cref{operators-on-U}. When restricted to $U$ they satisfy 
$\Theta_{i,U}=H^{i-1}_2 \tilde{\theta}_i$. 
\end{defn}

Now we extend $\tilde{\theta}_4$ to $\flag$ (the map has no poles along $D_1$). Importantly,
on $\flag$ the map 
$
\omega \hookrightarrow H \xrightarrow{\nabla_{\flag/\Shbar}} H \otimes \Omega^1_{\flag/\Shbar}
$
factors through $\omega \otimes \Omega^1_{\flag/\Shbar}$, since by functoriality of 
$\nabla$, every section that is a pullback from $\Shbar$
is horizontal for $\nabla_{\flag/\Shbar}$. 
This allows us to extend $\tilde{\theta}_4$ on $\LL(0,1)$ as
\[
\theta_4 : \LL(0,1) \xrightarrow{\nabla} \omega \otimes \Omega^1_{\flag/\Shbar} 
\xrightarrow{\pi_{(0,1)} \otimes \text{id}}
\LL(p-1,1).
\]
For $\LL(1,0)$ we prove that
the map over $U$
\begin{equation} \label{theta_4_(1,0)}
\theta_4: \LL(1,0)=\omega/\LL \cong \LL' \xrightarrow{\tilde{\theta}_4} \LL' \otimes \LL(-1,1) \xrightarrow{H_2} 
\LL(1,0) \otimes \LL(p-1,0)
\end{equation}
extends to $\flag$, by computing a local expression. We use the notation of 
\Cref{notation-local-computation}. When restricting the canonical isomorphism $\Omega^1_{\flag/\Shbar}=\LL(-1,1)$
to $W$ and using the trivialization given by $\{e_1,e_2\}$, we have
$$
\Omega^1_{\flag/\Shbar}|W \cong \Omega^1_{\p^1_W/W}=\mathcal{O}(-2)_W \cong \LL(-1,1)_W,
$$ 
sending $\text{dT}$ to $ e^{-1}_2(e_1 +Te_2)$.
On $U$ a basis for  
$\mathcal{L}'$ is $e_2-\lambda(e_1+Te_2)$, where
$\lambda\coloneqq \frac{1}{\tilde{H}_2}\frac{d\tilde{H}_2}{dT}$. Then 
$f e_2 \in \omega/\LL$ for $f \in \mathcal{O}_{\A^1_W}$ is sent by $\theta_4$ to  
\begin{align*}
f e_2  &\mapsto f(e_2-\lambda(e_1+Te_2)) \mapsto H_2(e_2-\lambda(e_1+Te_2)) \otimes df \\
&-H_2f(e_1 \otimes d(\lambda)+e_2 \otimes d(\lambda T)) 
\mapsto(\tilde{H}_2 \frac{df}{dT}-f\frac{d\tilde{H}_2}{dT}) e^{p-1}_2(e_1+Te_2).
\end{align*}
We have used that $\nabla_{\flag/\Shbar}(e_i)=0$.
Similarly, $f (e_1+Te_2) \in \LL(0,1)$ is sent to
$$
f (e_1+Te_2) \mapsto (\tilde{H}_2 e^p_2 \otimes df+f \frac{d\tilde{H}_2}{dT}e^p_2 \otimes dT)=
(\tilde{H}_2 \frac{df}{dT}+f\frac{d\tilde{H}_2}{dT}) e^{p-1}_2(e_1+Te_2).
$$

\begin{defn}(Extending $\tilde{\theta}_4$ to $\flag$) Define maps on $\flag$
  \begin{align*}
  \theta_4 : & \LL(0,1) \to H \xrightarrow{\nabla_{\flag/\Shbar}} H \otimes \Omega^1_{\flag/\Shbar} \xrightarrow{\pi_{(0,1)} \otimes \text{id}}
  \LL(p-1,1) \\
  \theta_4: & \LL(1,0)=\omega/\LL \cong \LL' \xrightarrow{\tilde{\theta}_4} \LL' \otimes \LL(-1,1) \xrightarrow{H_2} 
\LL(1,0) \otimes \LL(p-1,0).
  \end{align*}
  The latter is defined on $U$ but it extends to $\flag$ as explained above. 
  For $\lambda=0$ it is defined as $H_2d_{\flag/\Shbar}: \mathcal{O}_{\flag} \to \Omega^1_{\flag/\Shbar} \otimes \LL(p,-1)$,
  and for arbitrary $\lambda \in X^*(T)$ they are defined by the Leibniz rule. This gives maps
  $\theta_4: \LL(\lambda) \to \LL(\lambda+(p-1,0,p-1))$ on $\flag$ satisfying 
  $\theta_{4\mid U}=H_2 \tilde{\theta}_4$, and they are Hecke equivariant away from $p$.
  \end{defn}

  Putting together the local expressions from before we obtain the following. 

  \begin{lemma} \label{theta_4-local-expression} 
    Let $f e^k_2(e_1+Te_2)^l$ be a local section of $\LL(k,l)$ with $f \in \mathcal{O}_{\A^1_{W}}$. Then 
    $$
    \theta_4(f e^k_2 (e_1+Te_2)^l)=(\tilde{H}_2 \frac{df}{dT}+(l-k)f\frac{d\tilde{H}_2}{dT} ) e^{k+p-1}_2 (e_1+Te_2)^l.
    $$
    \end{lemma}

We now prove some properties about the theta operators $\Theta_i$ that can be checked on an open dense subset, 
so that they will automatically extend to the operators on $\flag$.
We use the convention $\Theta_4=\theta_4$.

\begin{prop}
  Let $i=1,2,3,4$, $\lambda, \mu \in X^{*}(T)$. 
  \begin{enumerate}
  \item $\Theta_i(fg)=f \Theta_i(g)+ \Theta_i(f)g$ for $f, g$ local sections of 
   $\LL(\lambda)$ and $\LL(\mu)$. 
  \item  $\Theta_i(H_j)=0$ for $j=1,2$.
  \end{enumerate}
  \end{prop}
  \begin{proof}
  The first one is automatic from the construction of $\Theta_i$.
  We check the second over $U_{1,\Ig}$, let $\{e_1,e_2\}$ be a canonical basis of $\omega_\Ig$.
  In this basis
  \begin{align*}
  H_1&=e^{p-1}_2 (e_1+Te_2)^{p-1} \\
  H_2&=(T-T^p) e^p_2  (e_1+Te_2)^{-1}.
  \end{align*}
  Using that $V \nabla(e_i)=\nabla^{(p)}(e^{(p)}_i)=0$ we see that both $e_2 \in \LL(1,0)$ and 
  $e_1+Te_2 \in \LL(0,1)$ are horizontal for the map $\theta_L$ in the definition of $\Theta_{1,2,3}$. 
  It is then clear that $\Theta_{1,2,3}(H_i)=0$ for $i=1,2$, and $\Theta_4(H_1)=0$. Finally, 
  $\Theta_4(H_2)=0$ from a direct computation with \Cref{theta_4-local-expression}.
  
  \end{proof}

  Let $s \in \Shbar^{\text{ord}}(\overline{\F}_p)$, and let  $\{e_1,e_2\}$ be a Serre--Tate
  basis of $\omega$ over $\widehat{\mathcal{O}}_{\Shbar,s}$, so
  that $V e_i=e^{(p)}_i$. Following the recipe for their construction gives the following expressions 
  for $\Theta_i$ on the fiber of $\flag$ over $\widehat{\mathcal{O}}_{\Shbar,s}$
  \begin{align*}
  \Theta_1&(f e^{k-l}_2 (e_1 \wedge e_2)^l) =\\
  &(T^2D_{11}(f)-TD_{12}(f)+D_{22}(f))e^{k-l+2}_2 (e_1 \wedge e_2)^{l} \\
  \Theta_2 &(f e^{k-l}_2 (e_1 \wedge e_2)^l)= \\
  &(2T^{p+1}D_{11}(f)-(T+T^p)D_{12}(f)+2D_{22}(f))e^{k-l+p+1}_2 (e_1 \wedge e_2)^{l} \\
  \Theta_3&(f e^{k-l}_2 (e_1 \wedge e_2)^l)=\\
  &(T^{2p}D_{11}(f)-T^p D_{12}(f)+D_{22}(f))e^{k-l+2p}_2 (e_1 \wedge e_2)^{l},
  \end{align*}
  where $D_{ij}\coloneqq \text{ks}(e_i e_j)^{\vee}=(T_{ij}+1)\frac{\partial}{\partial T_{ij}}$
  by \Cref{ST-coordinates}. 
  We have used that the canonical basis on $\omega_{\Ig}$ specializes to a Serre--Tate basis on  
  $\widehat{\mathcal{O}}_{\Shbar,s}$.
  Using these, we can prove that the following relations hold among the $\Theta_i$.
\begin{prop} \label{relations-U1}
The following relations hold as operators over $U_1$.
\begin{enumerate}
\item $[\Theta_1,\theta_4]\coloneqq \Theta_1 \circ \theta_4- \theta_4 \circ \Theta_1=\Theta_2$.
\item $[\Theta_2,\theta_4]=2\Theta_3$.
\item $[\Theta_3,\theta_4]=0$.
\item $[\Theta_{i}, \Theta_j]=0$ for $i,j=1,2,3$. 
\item $\Theta^p_1=\Theta_3$.
\item The operator $\Theta\coloneqq \frac{1}{H^2_2}(4\Theta_1 \Theta_3- \Theta^2_2)$ is well-defined and satisfies
 $\Theta^p=H^{2}_1\Theta$.
\end{enumerate}
\end{prop}
\begin{proof}
All of them can be checked locally
on fibers $\p^1_{\mathcal{O}_{\Shbar,s}}$ for $s \in \Shbar^{\text{ord}}(\overline{\F}_p)$. 
We can further base change to $\widehat{\mathcal{O}}_{\Shbar,s}$, since the map to the formal 
completion is faithfully flat. There they follow from a simple computation with the local expressions above,
and \Cref{theta_4-local-expression} with $\tilde{H}_2=T-T^p$.
\end{proof}

 \section{Extension to the flag Shimura variety} \label{Section3}

We extend $\Theta_{1,2,3,4}$ to maps on $\flag_{\overline{\F}_p}$,
by studying their formal local expressions at non-ordinary points. These local expressions
will imply a divisibility criterion 
for $H_1$ and $\theta_1$, which in turn will allow us to study the theta cycle for $\theta_1$ which defines 
the operator in \Cref{geometric-construction-the-map} inducing the generic entailment.  
 We also give alternative descriptions 
for $\pi_* \theta_{1,2,3}$ as maps on the Shimura variety.

\subsection{Extending $\Theta_{1,2,3}$ to $\flag$}
We extend $\Theta_{1,2,3}$ to $\flag$ by computing their poles along $D_1=\pi^{-1}(\Shbar^{\text{n-ord}})$. 
The next lemma contains the information to do this for $\lambda=0$.
\begin{lemma} \label{weight-0}
The section $s: \Omega^1_{\flag} \to \pi^*\Omega^1_{\Shbar}$ defined over $U_1$ extends to $\flag$ 
by postcomposing with $H_1: \pi^*\Omega^1_{\Shbar} \to \pi^*\Omega^1_{\Shbar} \otimes \LL(p-1,p-1)$.
\begin{proof}
Recall the $M \cap B$-reduction $I_{M \cap B} \subseteq I_{B}$ defined over $U_1$ by the splitting of the Hodge filtration. 
There exists an $M \cap B$-equivariant section $r: (\mathfrak{g}/\mathfrak{b})^{\vee} \to (\mathfrak{g}/\mathfrak{p})^{\vee}$, 
so that by \Cref{splitting-2} $s$ is defined by mapping $(\psi,x) \in 
F_{M \cap B}(\mathfrak{g}/\mathfrak{b})^\vee=\Omega^1_{U_1}$ with $\psi \in I_{M \cap B}$ 
to $(\psi,r(x)) \in \pi^*\Omega_{\Shbar^\text{ord}}$. 
Let $(\phi,x) \in \Omega^1_{\flag}$. Over $U_1$ there exists $g \in U_{P}(U_1)$ such that $g \phi \in I_{M \cap B}$. 
Then $s$ over $U_1$ sends $(\phi,x)=(g\phi,gx)$ to $(g\phi,r(gx))=(\phi,g^{-1}r(gx))=(\phi,r(gx))$,
the last equality since $U_{P}$ acts trivially on $\mathfrak{g}/\mathfrak{p}$. Thus, we want to prove that 
 $H_1g$ extends to $\flag$. Let $g_0 \in \Hom(L^{\vee} \otimes \mathcal{O}_{U_1}, L \otimes \mathcal{O}_{U_1})$ be induced by $g$,
 it is defined as the composition $g_0:  L^{\vee} \otimes \mathcal{O}_{U_1} \xrightarrow{\phi^{-1}} H/\omega \xrightarrow{\varphi^{-1}} \Ker(V)
 \xrightarrow{\phi} (L \oplus L^{\vee}) \otimes \mathcal{O}_{U_1} \to L \otimes \mathcal{O}_{U_1}$, 
 where $\varphi$ is the natural map $\Ker(V) \to H/\omega$. Since the determinant of $\varphi$ defines $H_1$ we see 
 that $H_1 \varphi^{-1}$, and hence $H_1g$ extend to $\flag$.
\end{proof}
\end{lemma}
Then we prove that the following common piece of $\Theta_{1,2,3}$ only has a simple pole.
\begin{lemma} \label{extend-s}
The map on $U_1$
 $$
 \theta\coloneqq H_1 \tilde{\theta}: \omega \xrightarrow{\nabla} H \otimes \Omega^1_{\flag} \xrightarrow{H_1 \pi_{\omega} \otimes s}
\omega \otimes \LL(p-1,p-1) \otimes \pi^* \Omega^1_{\Shbar},
$$ 
where $\tilde{\theta}$ is as in \Cref{magic-on-U1},
extends to  $\flag$. Moreover, on sections that come from $\omega/\Shbar$ via pullback it can be identified 
 with the map on $\Shbar$
$$
\omega \xrightarrow{\nabla} H \otimes \Omega^1_{\Shbar} \xrightarrow{H_1 \pi_\omega \otimes \text{id}} 
\omega \otimes \textnormal{det}^{p-1}\omega \otimes \Omega^1_{\Shbar}.
$$
\begin{proof}
The second statement is immediate from the first by functoriality of $\nabla$.
We prove that $\theta$ extends to $\flag$ by computing a local expression, we use the notation from  
\Cref{ig-coordinates}.
Let $W=\text{Spec}(R) \subset \Shbar$ such that $\omega$ is trivial with basis 
$\{e_1,e_2\}$, and work over $\A^1_W \subset \pi^{-1}(W)$, then $U_1$ is given by $R[1/(ad-bc)][T]$.
Let $d$ denote the differential $d: \mathcal{O}_{\pi^{-1}(W)} \to \Omega^1_{\pi^{-1}(W)}$. Let 
$f \in R[T]$, we compute that $H_1\tilde{\theta}(f e_{i})$ has no poles. 
On $\Ig_{W}$ we have the canonical basis $\{\tilde{e}_1,\tilde{e}_2\}$ satisfying 
$M(e_1,e_2)^{T}=(\tilde{e_1},\tilde{e}_2)^T$, which is horizontal for $\tilde{\theta}$.
Differentiating the relation $M^{(p)}=M V^{T}$ we obtain $dM=-M \circ d(V^{T}) \circ (V^{T})^{-1}$, so that 
\begin{equation} \label{eq234}
(\pi_\omega \otimes 1)\circ \nabla(e_1,e_2)^{T}=d(M^{-1})M(e_1,e_2)^{T}=d(V^T) \circ (V^{T})^{-1}(e_1,e_2)^{T}.
\end{equation}
This shows that the map $(H_1\pi_{\omega} \otimes 1) \circ \nabla$ extends to $\flag$. Concretely, it sends 
$fe_1$ to 
$$
fe_1 \mapsto [(ad-bc)df+f(dd(a)-bd(c))]e_1(e_1 \wedge e_2)^{p-1}+
f(ad(c)-cd(a))e_2(e_1 \wedge e_2)^{p-1},
$$
and similarly for $fe_2$.

Let $W'=W \cap \Shbar^{\text{ord}}$, the last ingredient of $\theta$ involves the section $s: \Omega^1_{\pi^{-1}(W')} \to \pi^* \Omega^1_{W'}$ 
which is defined by the diagram 
$$
\begin{tikzcd}[
  ar symbol/.style = {draw=none,"#1" description,sloped},
  isomorphic/.style = {ar symbol={\cong}},
  equals/.style = {ar symbol={=}},
  ]
\pi^{-1}\mathcal{O}_{W'} \arrow[r,hook] & \mathcal{O}_{\pi^{-1}(W')} \arrow[r,"d"] & \Omega^1_{\pi^{-1}(W')} \arrow[r,"s"] & \pi^* \Omega^1_{W'} \\
\pi^{-1}\mathcal{O}_{W'} \arrow[r,hook] \arrow[u,equal]& \mathcal{O}_{\mathbb{P}^1_{W'}} \ar[u,isomorphic] \arrow[r,"d"] & \Omega^1_{\mathbb{P}^1_{W'}}
 \ar[u,isomorphic] \arrow[ur,"r"],
\end{tikzcd}
$$
where the isomorphisms are given \'etale locally by the basis $\{\tilde{e}_1,\tilde{e}_2\}$,
and $r$ is the natural projection. 
Since the first square commutes for $g\in \pi^{-1}\mathcal{O}_W$
we have $s \circ d(g)=d_{W'}(g) \in \pi^* \Omega^1_{W'}$, where $d_{W'}: \pi^{-1}\mathcal{O}_{W'}  
\to \pi^{-1}\Omega^1_{W'}$.
Thus, to compute $\theta$, for $x \in \pi^{-1}\mathcal{O}_W$
 we replace $d(x)$ by $d_{W}(x)$ in \eqref{eq234}, and \Cref{weight-0} tells us that $(ad-bc)s \circ df$
 extends to $\flag$. Explicitly,
\begin{equation} \label{local-expression-1}
\theta(fe_1)=[(ad-bc)s \circ df+f(dd(a)-bd(c))]e_1(e_1 \wedge e_2)^{p-1}+f(ad(c)-cd(a))e_2(e_1 \wedge e_2)^{p-1},
\end{equation}
where now the differentials except $df$ correspond to $d: \mathcal{O}_W \to \Omega^1_W$. This proves that 
the map extends to $\flag$. 
\end{proof}
\end{lemma}

The local expression in \eqref{local-expression-1} is easier to compute using the map on $\Shbar$. 
This lemma is just a careful check that the construction using the section $s$ is compatible with 
the map on $\Shbar$.

\begin{cor}
The maps $H_1 \cdot \Theta_{1,2,3}$ extend to maps of sheaves on $\flag$. 
\begin{proof}
 By \Cref{magic-on-U1} the map $\theta$ from \Cref{extend-s} sends $\LL$ to $\LL \otimes \LL(p-1,p-1) \otimes \pi^* \Omega^1_{\Shbar}$, being 
 a condition that can be checked on an open dense subset. Then $H_1 \cdot \Theta_{1,2,3}$ can be constructed
 on $\flag$ using the same definition as for $\Theta_{1,2,3}$, but replacing $\tilde{\theta}$ by $\theta$.
\end{proof}
\end{cor}

However, multiplying by $H_1$ might not be 
necessary for some of the operators. We will show that $\Theta_{3}$ already extends to $\flag$.
To do that we compute a local expression for $\Theta_3$ on the 
formal neighbourhood of a $p$-rank $1$ point. 
For the purposes of checking poles along $D_1$ it is enough to check them inside $\Shbar^{\ge 1}=\Shbar^{\text{ord}} \cup \Shbar^{=1}$,
since by Hartogs lemma a function on $U_1$ that extends to $\flag_{\Shbar^{\ge 1}}$ already extends to $\flag$. 
Being a local problem, we will work on formal completions $\widehat{\mathcal{O}}_{\Shbar,q}$ (or their pullbacks to $\flag$)
for a fixed $q \in \Shbar^{=1}(\overline{\F}_p)$, and use Grothendieck--Messing theory to get a convenient basis for $\omega$.
Write  $\widehat{O}_{\Shbar,q}=\overline{\F}_p[[T_{11},T_{12},T_{22}]]$, and 
let $R=\mathcal{O}_{\Shbar,q}/ \mathfrak{m}_{q}^{p}$. As in the proof of \Cref{Hasse-simple-zero}
we can choose a basis of $\omega_{R}$ such that $V$ is represented 
by the matrix \begin{equation} \label{V-prank=1}
  V_R=\begin{pmatrix}
  1 & 0 \\
  T_{12} & T_{22}
  \end{pmatrix}.
  \end{equation}
Let $q \in \Shbar^{=1}(\overline{\F}_p)$, and fix a basis $\{e_1,e_2\}$ of $\omega$ on $\widehat{O}_{\Shbar,q}/\mathfrak{m}^{p}_q$ 
such that $V$ has matrix \eqref{V-prank=1}. Consider $D_{ij}\coloneqq \textnormal{ks}(e_i e_j)^{\vee}$ as derivations on 
$R$.

\begin{lemma} \label{KS-modulo}
  Let $q \in \Shbar^{=1}(\overline{\F}_p)$. Then 
  $$
  D_{ij}= \frac{\partial}{\partial T_{ij}}  \bmod{\; \mathfrak{m}_{\Shbar,q}^{p-1}}.
  $$
  \end{lemma}
  
  \begin{proof}
  Dually, we prove that $\langle \nabla(e_i),e_j \rangle=d T_{ij} \bmod{\; \mathfrak{m}_{\Shbar,q}^{p-1}}$. 
  Let $R=\mathcal{O}_{\Shbar,q}/\mathfrak{m}_{\Shbar,q}^{p}$,
  then
   $$
   \langle \nabla(e_i),e_j \rangle=\langle \nabla_{R}(e_{R,i}), e_{R,j} \rangle
    \bmod{ \;\mathfrak{m}_{\Shbar,q}^{p-1}}
   $$
   since $\nabla_R=\nabla \bmod{\; \mathfrak{m}_{\Shbar,q}^{p}}$ determines $\text{ks}
   \bmod{\; \mathfrak{m}_{\Shbar,q}^{p-1}}$. In the notation of \eqref{omega_R}
   \begin{align*}
   e_{R,1}&=\overline{e}_1+T_{12}\overline{e}_3+T_{11}\overline{e}_4 \\
   e_{R,2}&=\overline{e}_2+T_{22}\overline{e}_3+T_{12}\overline{e}_4,
   \end{align*}
   where $\{\overline{e}_i\}$ is a symplectic basis on $H^1_{\kappa(q)}$ coming 
   from a trivialization of $A[p]_{\kappa(q)}$. By \Cref{gm-gauss-manin}
   $\nabla_R(\overline{e}_i)=0$. Then 
   $$
   \langle \nabla(e_1),e_1 \rangle=\langle dT_{12} \overline{e}_3+ dT_{11} \overline{e}_4, \overline{e}_1 \rangle=
   dT_{11} \bmod{\; \mathfrak{m}_{\Shbar,q}^{p-1}},
   $$
   and similarly for the others. 
  \end{proof}

  Using the lemma above and the local expression \eqref{local-expression-1} we can compute the following local expressions for 
  $H_1 \Theta_{1}$ and $H_1 \Theta_{3}$.
  \begin{prop} \label{full-local-expression}
  Let $q \in \Shbar^{=1}(\overline{\F}_p)$, and $\{e_1,e_2\}$ be a basis of $\omega$ on $\widehat{O}_{\Shbar,q}/\mathfrak{m}^{p}_{q}$
  satisfying \eqref{V-prank=1}. 
  We have the following local expressions at $\mathbb{A}^1_{T} \subset \p^1_{\widehat{\mathcal{O}}_{\Shbar,q}}$ with respect 
to the local sections $f e^{k-l}_2 (e_1 \wedge e_2)^l \in \LL(k,l)$ for $f \in \widehat{\mathcal{O}}_{\Shbar,q}[T]$,
and $g e^i_1e^{k-l-i}_2(e_1 \wedge e_2)^l$
 for $g \in \widehat{\mathcal{O}}_{\Shbar,q}$, $i\in \Z$.
\begin{align*}
H_1\Theta_1(f)&=kf \bmod{\mathfrak{m}_{q}}\\
H_1\Theta_3(g)&=0  \bmod{\mathfrak{m}_{q}}
\end{align*}

\begin{proof}
Using the explicit matrix of $V$ in \eqref{V-prank=1} on
the local expressions for $\theta$ in \Cref{extend-s}, as well as in \eqref{projections-from-KS};
and the knowledge of Kodaira--Spencer modulo
$\mathfrak{m}^{p-1}_{q}$ of \Cref{KS-modulo} we get 
\begin{align*}
  H_1\Theta_{1}(e^{k-l}_2 (e_1 \wedge e_2)^l)&=k \bmod{\mathfrak{m}_{q}^{p-1}} \\
H_1\Theta_{3}(e^i_1e^{k-l-i}_2(e_1 \wedge e_2)^l)&=(k-i)T^2_{22} e^i_1 e^{k-l-i+2p}_2 (e_1 \wedge e_2)^{l+p-1} \\
&+i T_{22}T_{12} e^{i-1}_1 e^{k-l-i+2p+1}_2 (e_1 \wedge e_2)^{l+p-1} \\
&-T_{22}ie^{i+p-1}_1 e^{k-l-i+p+1}_2 (e_1 \wedge e_2)^{l+p-1} \bmod{\mathfrak{m}_{q}^{p-1}}.
\end{align*}
To conclude we check that $\Theta_{1}$ and $\Theta_{3}$ extend to $\flag$ when applied to $\mathcal{O}_{\flag}$. 
It is enough to check this for $\Pi_{1,3}$. We use \Cref{splitting-2} to identify $\pi_{1,3}$ over $U$ with maps 
of $T$-representations. For $\Pi_1$ it follows from the fact that $-\beta \to \mathfrak{g}/\mathfrak{b}$ is a map of $B$-representations. 
For $\Pi_3$, we want to check that $\pi_3$ already extends to $U_2$. Let $\tilde{M}$ be the Levi of $G$
 associated to
$\{\beta\} \subseteq \Delta$. Then over $U_2$ there is a $\tilde{M} \cap B$ reduction 
$I_{\tilde{M} \cap B} \subseteq I_{B}$ by considering the splitting $\LL \oplus V^{-1}(\LL^{(p)})=H$. Then the map 
$-2\alpha-\beta \to \mathfrak{g}/\mathfrak{b}$ giving raise to $\pi_3$ over $U$ is also $\tilde{M} \cap B$-equivariant. 
\end{proof}
  \end{prop}

  It is then clear that $\Theta_3$ already extends to $\pi^{-1}(\Shbar^{\ge 1})$, since it does pointwise,
  and hence to $\flag$. 
  Including the central character in the Kodaira--Spencer isomorphism 
  we get Hecke equivariant maps. 
  \begin{prop} \label{weights-thetas}
  Define $\theta_{1,2,3}$ on $\flag$ by 
  \begin{align*}
  \theta_1&\coloneqq H_1 \Theta_1: \LL(k,l,c) \to \LL(k+p+1,l+p-1,c) \\
  \theta_2&\coloneqq H_1 \Theta_2: \LL(k,l,c) \to \LL(k+2p,l+p-1,c) \\
  \theta_3&\coloneqq \Theta_3: \LL(k,l,c) \to \LL(k+2p,l,c) \\
   \theta_4&: \LL(k,l,c) \to \LL(k+p-1,l,c+p-1).
  \end{align*}
  They are Hecke equivariant away from $p$, and they satisfy relations as in \Cref{relations-U1}, namely
  \begin{enumerate}
    \item $[\theta_1,\theta_4]=\theta_2$.
    \item $[\theta_2,\theta_4]=2H_1\theta_3$.
    \item $[\theta_3,\theta_4]=0$.
    \item $[\theta_{i},\theta_j]=0$ for $i,j=1,2,3$. 
    \item $\theta^p_1=H^p_1\theta_3$.
    \item The operator $\Theta\coloneqq \frac{1}{H^2_2}(4H_1\theta_1 \theta_3- \theta^2_2)$ is well-defined and satisfies
     $\Theta^p=H^{2p}_1\Theta$.
    \end{enumerate}

  \end{prop}

 Looking at the local expression for $\theta_1$ we immediately obtain a criterion for the divisibility 
 of $\theta_1$ by $H_1$.

\begin{prop} \label{divisibility-properties}
  Let $f \in \LL(k,l)$ be a local section such that $H_1 \nmid f$. 
  Then 
  $H_1 \mid \theta_1(f)$ if and only if
  $p \mid k$. 
  \begin{proof}
  It follows from \Cref{full-local-expression}, since if $H_1 \nmid f$ we can take 
  $q \in \Shbar^{=1}(\overline{\F}_p)$ so that 
   $\tilde{H}_1=T_{22}$ doesn't divide $f$
  in these coordinates. Then in these coordinates $H_1 \mid \theta_1(f)$ implies $p \mid k$.
  Conversely, if $T_{22} \mid \theta_1(f)$ for all $q \in \Shbar^{=1}(\overline{\F}_p)$ 
  then $H_1 \mid \theta_1(f)$. 
  \end{proof}
  \end{prop}

  Combining \Cref{divisibility-properties} with the identity $\theta^p_1=H^p_1 \theta_3$ 
  we obtain the map that produces the generic entailment, similarly to the way one determines the theta cycle
  for the modular curve. 

  \begin{theorem} \label{geometric-construction-the-map}
  Let $(k,l) \in X^*(T)$, and write $k=pb+a$ with $1 \le a \le p$. Then the map 
  $$
  \theta^1_{(k,l)}\coloneqq \frac{1}{H^{p-a+1}_1}\theta_1^{p-a+1} : \LL(k,l) \to \LL(2p-2a+k+2,l)
  $$
  exists as a map of sheaves on $\flag$, and it is Hecke equivariant away from $p$.
 
 \begin{proof}
 We prove the statement by induction on $a$, starting at $p$, then $p-1$, $p-2$ and so on. For $f \in \LL(k,l)$
 we write $a(f)$ to emphasize that it depends on $f$.
 Let $f \in \LL(k,l)$ be a local section over some open $V \subseteq \flag$. We want to prove that 
 $H^{p-a+1}_1 \mid \theta^{p-a+1}_1(f)$.
 If $V \subset U_1$ then the assertion is trivial,
 suppose it's not. 
 First assume $a=p$, if 
 $H_1 \mid f$ the statement is obvious, otherwise it follows from \Cref{divisibility-properties}. Now assume 
 $a \neq p$. If $H_1 \mid f$ write $f=H_1 f_0$, then the statement $H^{p-a+1}_1 \mid 
 \theta^{p-a+1}_1(f)$ follows from $H^{p-a(f_0)+1}_1 \mid 
 \theta^{p-a(f_0)+1}_1(f_0)$ for $a(f_0)=a(f)+1$, which appears earlier in the induction process.
 Therefore, we may assume $H_1 \nmid f$. 
 By \Cref{divisibility-properties} we see that 
 $H_1 \nmid \theta^{p-a}_1(f)$ but $H_1 \mid \theta^{p-a+1}_1(f)$. Let $N$ be the exact power of $H_1$ 
 that divides $\theta^{p-a+1}(f)$. We can assume that $a>1$ and $1 \le N \le p-1$, otherwise the statement 
 follows from $\theta^p_1=H^p_1 \theta_3$.
 Then $h\coloneqq \frac{1}{H^N_1} \theta^{p-a+1}_1(f)$ satisfies $a(h)=N+1$.
  We know that $H^{p-N}_1 \mid \theta^{a-1}_1(h)$
 by the relation 
 $\theta^p_1=H^p_1 \theta_3$, so for some $0 \le i \le a-2$, the $a(\theta^i(h))$ has to be divisible 
 by $p$ again. This implies that $(N+1)+a-2 \ge p$, i.e. $N \ge p-a+1$. 
 \end{proof}
  \end{theorem}

 We will
 prove that for generic $p$-restricted weights the map
 is injective on global sections in \Cref{geometric-entailment}.

\subsection{$\pi_* \theta_{1,2,3}$ as maps on the Shimura variety.}

Here we describe $\pi_* \theta_{1,2,3}$ without referring to the flag Shimura variety. 
First we recall the theta operator on $\Shbar$ already defined in the literature.
On $\Shbar^{\text{ord}}$ we have the projection $\pi_\omega
: H \to \omega$, which can be extended to a map  $\Pi_{\omega}: H \to \omega \otimes \det^{p-1} \omega$
on $\Shbar$ by multiplying 
by the Hasse invariant. We use $H(k,l)$ to denote the sheaf $\Sym^{k-l} H \otimes (\bigwedge^2 H)^{\otimes l}$
over $\Shbar$, the Gauss-Manin connection $\nabla$ extends to it. The Hodge filtration on $H$
defines a decreasing filtration, and by Griffiths transversality  $\nabla$ on $H(k,l)$ factors through the first non-trivial
step of the filtration $F^1 H(k,l)$. Then $\Pi_{\omega}$ extends to a map $\Pi: F^1 H(k,l) \to \omega(k,l) \otimes \det^{p-1} \omega$
such that on a simple tensor one only applies $\Pi_{\omega}$ to the component which doesn't lie in $\omega$ if there is one,
otherwise $\Pi$ acts as multiplication by $H_1$. 
\begin{defn}
Let $(k,l) \in X^*(T)$ be a $M$-dominant weight. 
Define the differential operator $\theta_{\Shbar}: \omega(k,l) \to \omega(k+p-1,l+p-1) \otimes \Sym^2 \omega$
as the composite of the following maps
$$
\theta_{\Shbar}\coloneqq \omega(k,l) \xrightarrow{\nabla} F^1 H(k,l) \otimes \Omega^1_{\Shbar/\F_p} \xrightarrow{\Pi \otimes \text{ks}^{-1}}
\omega(k+p-1,l+p-1) \otimes \Sym^2 \omega.
$$
For $(k,l)=(0,0)$ by convention $\nabla$ becomes $d: \mathcal{O}_{\Shbar} \to \Omega^1_{\Shbar}$.
\end{defn}

Recall the $3$ projections from $\Sym^2\omega $ on $\flag$ in \eqref{projections-from-KS}.
Similarly, we can also define $3$ projections from $\Sym^2 \omega$ over $\Shbar$:
\begin{align*}
\Pi_{\Shbar,1}: & \Sym^2 \omega \xrightarrow{\text{id}} \Sym^2 \omega \\
\Pi_{\Shbar,2}: & \Sym^2 \omega \to \omega^{\otimes 2} \xrightarrow{\text{id} \otimes V} \omega \otimes \omega^{(p)}
\to \omega \otimes \Sym^p \omega \to \Sym^{p+1} \omega \\
\Pi_{\Shbar,3}: & \Sym^2 \omega \xrightarrow{\Sym^2 V} \Sym^2 \omega^{(p)} \to \Sym^{2p} \omega,
\end{align*}
where we are using the natural embedding $\omega^{(p)} \hookrightarrow \Sym^p \omega$. 

\begin{lemma} \label{identifications-projections}
Under the identifications $\pi_* \LL(k,l)=\omega(k,l)$, we have $\pi_* \Pi_i= \Pi_{\Shbar,i}$. 
Moreover, the pushforward of 
$$
\LL(k,l) \otimes \Sym^2 \omega \xrightarrow{\text{id} \otimes \Pi_i}
\LL(k+k_i,l+l_i)
$$
is 
$$
\omega(k,l) \otimes \Sym^2 \omega \xrightarrow{\text{id} \otimes \Pi_{\Shbar,i}}
\omega(k,l) \otimes \omega(k_i,l_i) \xrightarrow{r} \omega(k+k_i,l+l_i),
$$ where 
$r$ comes from the canonical projection $\Sym^n \otimes \Sym^m \to \Sym^{n+m}$.
\end{lemma}

\begin{proof}
 For $i=1$ the property that we need is that the map $\LL(k,l) \otimes \Sym^n \omega \to \LL(k+n,l)$
 corresponds to $r$ under $\pi_*$. This can be checked on local coordinates, reducing to the case of 
 $\mathbb{P}^1$. We also have that $\omega^{(p)} \to \LL(p,0)$ corresponds to $\omega^{(p)} \to 
 \Sym^p \omega$, also by reducing to the case of $\mathbb{P}^1$. The statement for $i=2,3$
 follows by combining these two properties and the fact that $\pi_*(V: \omega \to \omega^{(p)})$ is again
 the Verschiebung on $\Shbar$
 by the projection 
 formula. 
\end{proof}

We can use these to describe $\pi_* \theta_{1,2,3}$ as maps on $\Shbar$.

\begin{prop} \label{theta-on-X}
Let $(k,l) \in X^{*}(T)$ $M$-dominant. We define $3$ different theta operators as maps of sheaves on $\Shbar$:
\begin{align*}
\theta_{\Shbar,1}:&=r \circ \Pi_{\Shbar,1} \circ \theta_{\Shbar} : \omega(k,l) \to  \omega(k+p+1,l+p-1) \\
\theta_{\Shbar,2}:&=r \circ \Pi_{\Shbar,2} \circ \theta_{\Shbar}: \omega(k,l) \to \omega(k+2p,l+p-1) \\
\theta_{\Shbar,3}:&=r \circ \Pi_{\Shbar,3} \circ \theta_{\Shbar}: \omega(k,l) \to \omega(k+3p-1,l+p-1),
\end{align*}
where $r$ is the natural map  $\Sym^n \otimes \Sym^m \to \Sym^{n+m}$. Then 
$\pi_* H_1 \Theta_i=\theta_{\Shbar,i}$. 
\begin{proof}
By a Leibniz rule type of argument we can reduce to weights $(1,0)$, $(1,1)$ and $(0,0)$. 
\Cref{extend-s} shows that $\pi_{*} \theta=\theta_{\Shbar}$, which combined 
with \Cref{identifications-projections} proves it for weight $(1,0)$. For $(1,1)$ it follows similarly from
\Cref{extend-s} by following the recipe for $H_1 \Theta_i$ in the case 
$\text{det } \omega= \LL \otimes \omega/\LL$. For $(0,0)$ it follows from the diagram in the proof of \Cref{extend-s}.
\end{proof}
\end{prop}

The operator $\theta_{\Shbar,1}$ is the theta operator already considered by \cite{Eischen-Mantovan-1} in 
much greater generality,
and by \cite{Yamauchi-1} (denoted by $\theta_3$ in his paper) for $\gsp_4$. 
Since $\Theta_3$ already extends to $\flag$ we see that $H_1 \mid \theta_{\Shbar,3}$. One could have 
proved this via a local expression directly on $\Shbar$.

With this description we can prove the following important result.
\begin{theorem} \label{geometric-entailment}
  Let $(k,l) \in X^*(T)$ such that $p-1 \ge k \ge l \ge 0$, then
  $\Ker \pi_*\theta^1_{(k,l)}=\Ker \pi_* \theta_1$. If furthermore $(k,l) \neq (p-1,p-1)$ 
  $$
  \theta^1_{(k,l)}: \H^0(\Shbar,\omega(k,l)) \to \H^0(\Shbar,\omega(2p-k+2,l)) 
  $$
  is injective. The map $\theta^1_{(k,l)}$ is still injective for $(k,l)=(p-1,p-1)$ after localizing
   at a non-Eisenstein maximal ideal.
  \begin{proof}
  We have an inclusion $\Ker \pi_* \theta_1 \subset \Ker \pi_*\theta^1_{(k,l)}$ since 
  $H^{p-k+1}_1 \theta^1_{(k,l)}=\theta^{p-k+1}_1$ and $H_1$ is injective. Also, $\Ker \pi_*\theta^1_{(k,l)}
  \subset \Ker \pi_* \theta_3$ since $\theta_3=\theta^1_{(2p-k+2,l)} \circ \theta^1_{(k,l)}$.
   We prove first that whenever $0\le k-l\le p-1$
   the map $f:\omega(k,l) \otimes \Sym^2 \omega \xrightarrow{\text{id} \otimes \Pi_{\Shbar,3}}
  \omega(k,l) \otimes \Sym^{2p}\omega \to \omega(2p+k,l)$ 
  is injective as a map of sheaves.
  It is enough to do this on $\Shbar^{\text{ord}}$, which is open dense, and 
  one can further pass to the finite \'etale cover $\Ig$. The canonical basis $\{e_1,e_2\}$ of $\omega$
  satisfies $Ve_i=e^{(p)}_i$, so $f$ is identified with the map 
  $\Sym^{k-l} \mathcal{O}^2 \otimes \Sym^2 \mathcal{O}^2 \to \Sym^{2p+k-l} \mathcal{O}^2$
  sending basis elements $e^n_1 e^{k-l-n}_2 \otimes e_i e_j$ to $e^n_1 e^{k-l-n}_2e^p_i e^p_j$,
   which is clearly injective when $k-l \le p-1$. Therefore $\Ker \pi_*\theta_3=\Ker \theta_{\Shbar,3}=
   \Ker(\theta_{\Shbar}: \omega(k,l) \to \omega(k+p-1,l+p-1) \otimes \Sym^2 \omega)$. Since the map 
   $\theta_{\Shbar}$ is a composition 
   factor of $\pi_* \theta_1=\theta_{\Shbar,1}$ we have $\Ker \pi_* \theta_3 \subset \Ker \pi_*\theta_1$, so that
   for $0\le k-l\le p-1$
   $$
   \Ker \pi_*\theta^1_{(k,l)}=\Ker \pi_* \theta_1.
   $$
  Thus, to prove injectivity on $\H^0$ it is enough to prove that $\theta_1$ is injective on global sections.
  Assume first $(k,l) \neq (p-1,p-1)$.
  Let $f \in \H^0(\Shbar,\omega(k,l))$ non-zero. It can't be divisible by $H_1$ since 
  $\H^0(\Shbar, \omega(k-p+1,l-p+1))=0$ 
  by \Cref{lan-suh}$(2)$. 
  Therefore, $\theta_1(f) \neq 0$ by 
  \Cref{divisibility-properties}. For $(k,l)=(p-1,p-1)$, note that $\H^0(\Shbar,\mathcal{O}_{\Shbar})$ 
   only has Eisenstein Hecke eigensystems (the trivial one),
  and $H_1$ is also Eisenstein, so by the same argument
   $\theta^1_{(k,l)}$ is injective after localizing at a non-Eisenstein maximal ideal
  $\mathfrak{m}$.
  \end{proof}
  \end{theorem}

  The proof of \Cref{geometric-entailment} also shows that $\theta^1_{(k,l)}$
  is injective on global sections for generic $p$-restricted weights $(k,l)$.

\section{The entailment of Serre weights} \label{Section5}

\subsection{Generalities about the weight part of Serre's conjecture}
We briefly recall the weight part of Serre's conjecture for $\gsp_4/\Q$ and how its combinatorics 
relates to the weight shifting of our operators. Let $\mathfrak{m} \subseteq \mathbb{T}$ be a mod $p$ 
Hecke eigensystem and $\overline{r}_{\mathfrak{m}}: G_{\Q} \to \gsp_4(\overline{\F}_p)$ its associated 
mod $p$ Galois representation. Let $\overline{\rho}_{\mathfrak{m}}: G_{\Q_p} \to \gsp_4(\Fpbar)$ be its restriction 
to the decomposition group at $p$.
We define the set of modular 
Serre weights as
$$
W(\overline{\rho}_{\mathfrak{m}}):=\{\sigma: \H^*_{\et}(\Sh_{\overline{\Q}_p},\sigma)_{\mathfrak{m}} \neq 0 \},
$$ 
where $\sigma$ runs along all Serre weights. 
As the notation indicates, conjecturally this only depends on $\overline{\rho}_{\mathfrak{m}}$.
We will often write $W(\overline{\rho})$ for $W(\overline{\rho}_{\mathfrak{m}})$, as it is customary in the literature. 
 We explain Herzig's \cite{Herzig1}
conjectural recipe for 
$W(\overline{\rho})$ in the case when $\overline{\rho}$ is semisimple, adapted to the case of $\gsp_4/\Q$ 
in \cite{Herzig-Tilouine}, and now proved in the latter case under 
some Taylor-Wiles and genericity conditions 
in \cite{heejong-lee}.\footnote{In \cite{heejong-lee} they use a different global setting, 
but their results can be translated to our setting using the vanishing results of \cite{Hamann-Lee}.} In general one expects that
$W(\overline{\rho}) \subseteq W(\overline{\rho}^{\text{ss}})$, so that the latter would exhaust all possible Serre weights. 
Given $\overline{\rho}: G_{\Q_p} \to \gsp_4(\overline{\F}_p)$ we can
write the semisimplification of the restriction to inertia $\tau\coloneqq \overline{\rho}^{\text{ss}}_{\mid I_p}$ as
$$
\tau=\tau(\mu, w)\coloneqq(\overline{\mu}+pw^{-1}\overline{\mu}+ \ldots 
p^{t-1}w^{1-t}\overline{\mu})\omega_{t}
$$
where $w \in W$, $t \in \{1,2,3,4\}$ is its order, $\omega_t$  is the fundamental 
character of niveau $t$, $\mu=(a,b,c) \in X^{*}(T)$, and $\overline{\mu}=(\frac{a+b+c}{2},
\frac{a-b+c}{2},c) \in X_{*}(\check{T})$. We are fixing
the isomorphism of $\gsp_4$ with its dual group which on Weyl groups we denote by $w \to \check{w}$, that switches the actions 
of $s_0$ and $s_1$. We then have $\check{w} \tau \sim \tau^p$. 
Then any semisimple $\tau: I_{\Q_p} \to \gsp_4(\overline{\F}_p)$ that lifts 
to $G_{\Q_p}$ (a tame inertial parameter)
can always be written as $\tau(\mu,w)$, and for generic $\tau$ we can assume $\mu \in X_1(T)$.
Let $\delta >0$, we say that a weight $\lambda \in X^*(T)$ is $\delta$-generic if 
 $\lvert \langle \lambda, \gamma^{\vee} \rangle -pn \rvert<\delta$ for each $\gamma \in \Phi$ 
 and some $n \in \Z$ depending on 
 $\gamma$. A tame inertial parameter $\tau=\tau(\mu+\rho,w)$ is $\delta$-generic if $\mu$ is.
We say that a statement holds for sufficiently generic weights/tame parameters if there exists some 
$\delta$ independent of $p$
such that the statement holds for all $\delta$-generic weights/tame parameters. 
The set of \textit{regular} weights $X_{\text{reg}}(T) \subset X_1(T)$ consists of all $(a,b) \in X_1(T)$
such that $a-b,b<p-1$.
Based on the reduction of the Deligne-Lusztig representation associated to $\tau$ Herzig 
constructs a set $W^{?}(\overline{\rho})$ of Serre weights, and conjectures that it should equal
the set of regular Serre weights when $\overline{\rho}$
is semisimple. 
Restricting to some set of sufficiently generic weights (or to sufficiently generic tame parameters)
$W^{?}(\overline{\rho})$
can also be described as
\begin{align*}
  W^{?}(\overline{\rho})=\{ F(\mu) \text{ with }\mu  \in X_{\text{reg}}(T) \text{ such that  } & \exists \; \mu' \uparrow \mu \text{ with } \mu'+\rho 
  \text{ dominant and} \\
  & w \in W,
   \text{ such that }
  \overline{\rho}^{\text{ss}}_{\mid I}
  =\tau(\mu'+\rho,w)\}
  \end{align*}
\cite[Prop 6.28]{Herzig1}. For $\overline{\rho}$ ordinary this description holds for all regular weights 
and in general $2$-generic is enough \cite[Prop 4.11]{Herzig-Tilouine}.
In the general case
$W^{?}(\overline{\rho})$ contains 
$20$ Serre weights, which 
can be classified into $8$ \textit{obvious} weights and $12$ \textit{shadow} weights, using the terminology of 
\cite{Gee-Herzig-Savitt}. The $8$ obvious 
weights correspond to taking $\mu'=\mu$ in the description of $W^{?}(\overline{\rho})$,
and the $12$ remaining shadow weights correspond to the cases where $\mu' \neq \mu$. 
The obvious weights match up with the $8$ "obvious" crystalline lifts \footnote{The lift only has to agree with $\overline{\rho}$
restricted to inertia.} of $\overline{\rho}$ with Hodge-Tate weights 
prescribed by
$\mu+\rho$, obtained by applying an element of $W$ to $\overline{\rho}$ and then using the congruences of the fundamental 
characters. We will see that the weights appearing in the generic entailment are shadow weights 
corresponding to a lower alcove weight.

\subsection{The generic entailment of Serre weights}

We first explain how to translate between coherent and \'etale cohomology. For $V \in \Rep_{\overline{\F}_p}G$ there is 
an associated $\gsp_4(\F_p)$-local system on $\Sh_{\Q_p}$ with the same name, by composing the local system
 $A[p]$ with $V_{\mid G(\F_p)}$. 
Generically for $p$-restricted weights we can compare the Hecke eigensystems appearing in
coherent and \'etale cohomology, but only when the coefficients for \'etale cohomology are dual Weyl modules. 
Recall that the spherical Hecke algebra
$\mathbb{T}/\Z_p$ acts both on coherent and \'etale cohomology, and we say that 
a maximal ideal $\mathfrak{m} \subset \mathbb{T}$ appearing in either cohomology is non-Eisenstein 
if the associated residual Galois representation is irreducible as a $\GL_4$ valued 
representation. 
We will also say that $\mathfrak{m}$ is generic if there exists an auxiliary prime $l \neq p$ such that the restriction of
$\mathfrak{m}$ to the Hecke algebra at $l$ is generic in the sense of \cite[Def 1.1]{Hamann-Lee}.
\begin{prop} \label{coherent-to-betti}
Let $\lambda=(a,b) \in X^*(T)$ be a dominant weight. Let $\mathfrak{m} \subset \mathbb{T}$ be a non-Eisenstein generic 
maximal ideal. Then
$$
\H^3_{\text{\'et}}(\Sh_{\overline{\Q}_p},V(a,b)_{\overline{\F}_p})_{\mathfrak{m}}=\H^3_{\text{\'et},c}(\Sh_{\overline{\Q}_p},V(a,b)_{\overline{\F}_p})_{\mathfrak{m}},
$$
$$
\H^{1}(\Sh^{\tor}_{\overline{\F}_p},\omega^{\can}(a+3,b+3))_{\mathfrak{m}}=\H^{1}(\Sh^{\tor}_{\overline{\F}_p},\omega^{\sub}(a+3,b+3))_{\mathfrak{m}},
$$
and
$$
\H^3_{\text{\'et}}(\Sh_{\Q_p},V(a,b) \otimes \overline{\F}_p)_{\mathfrak{m}} \neq 0  \implies 
\H^0(\Sh^{\tor}_{\overline{\F}_p},\omega^{\can}(a+3,b+3))_{\mathfrak{m}} \neq 0.
$$
Moreover, $\H^{*}_{\text{\'et}}(\Sh_{\Q_p},V(a,b)_{\overline{\F}_p})_{\mathfrak{m}}$ is concentrated 
in degree $3$.
Further, for $p \ge 5$ if $\lambda=(a,b) \in X_1(T)$ satisfies $ b \ge 3$,
or $b=2$ and $a \le p-2$, or $b=1$ and $a \le p-3$, then the reverse implication 
to the one above is also true. 
\begin{proof}
   The first equality follows from a similar (but simpler) argument to the one in \cite[Thm 4.2]{Newton-Thorne-CM}.
   To adapt their method 
  one needs to prove that the cohomology of the boundary of the Borel-Serre compactification 
  of $\Sh_{K}$ is Eisenstein. 
  The boundary is stratified by locally symmetric spaces associated to rational parabolics of $G$, whose Levi quotients 
  are respectively a split torus $T$ (Borel), $\GL_2 \times \mathbb{G}_m$ (Siegel), $\GL_2 \times \mathbb{G}^2_m$ (Klingen).
  One can attach mod $p$ Galois representations to mod $p$ eigensystems appearing in the cohomology of 
  the locally symmetric spaces for these Levi subgroups. In fact, in our case these locally symmetric spaces 
  are essentially the modular curve, where the construction of mod $p$ Galois representations is a classical result.  
  That way, as in the proof of \cite[Thm 4.2]{Newton-Thorne-CM}
  one can show that the Galois representation attached to 
  the Hecke eigenclasses appearing in the boundary are always a direct sum of 
  Galois representations attached to one of the proper rational Levis. Therefore, as a $\GL_4$-valued 
  representation $\overline{r}_{\mathfrak{m}}$ 
  is respectively the 
  sum of $4$ characters (Borel), the sum of two $2$-dimensional representations (Siegel),
  or the sum of $2$ characters and a $2$-dimensional representation (Klingen).  
  The second equality is proved by mimicking the proof of \cite[Thm 24.iii)]{Atanasov-Harris}, 
  which concerns unitary Shimura varieties. 
  Namely, the boundary of $\Shbar^{\tor}$ is also stratified by components labelled by rational parabolics.
  The Galois representations attached to these components are then summands of
  $\overline{r}_{\mathfrak{m}}$. We also note that
  as in the unitary case $\gsp_4$ is self-dual, so that their proof carries over to our setup. 

For the first implication about the existence of nontrivial eigensystems, 
  the vanishing theorem of \cite{Hamann-Lee} and the first equality imply that 
  $\H^{*}_{\et}(\Sh_{\overline{\Q}_p},V(a,b)_{\overline{\F}_p})_{\mathfrak{m}}$ is concentrated in degree $3$. 
  This means that we can lift any non-trivial mod $p$ eigensystem to 
  $\H^{3}_{\et}(\Sh_{\overline{\Q}_p},V(a,b)_{\check{\Z}_p})_{\mathfrak{m}}$. The latter is also torsion-free,
  so that there exists an ideal $\tilde{\mathfrak{m}} \subset \mathbb{T}$ corresponding 
  to an integral Hecke eigensystem in $\H^{*}_{\et}(\Sh_{\overline{\Q}_p},V(a,b)_{\check{\Z}_p})_{\mathfrak{m}}$.
  Then
  \begin{equation} \label{BGG}
  \H^3_{\text{\'et}}(\Sh_{\overline{\Q}_p},V(a,b)_{\overline{\Q}_p})_{\tilde{\mathfrak{m}}}=
  \oplus_{w \in W^{M}} \H^{l(w)}(\Sh^{\tor}_{\overline{\Q}_p},\omega(w \cdot (a,b)+(3,3)))_{\tilde{\mathfrak{m}}}
  \end{equation}
  by
  \cite{faltings-chai}, and the rational 
  \'etale-crystalline comparison theorem. Moreover, since the Galois 
  representation attached to $\tilde{\mathfrak{m}}$ is irreducible we have that the non-vanishing of the left-hand side 
  implies the non-vanishing of any of $\H^{l(w)}(\Sh^{\tor}_{\overline{\Q}_p},\omega(w \cdot (a,b)+(3,3)))_{\tilde{\mathfrak{m}}}$.
  This follows from the fact in the \'etale cohomology localized at such $\tilde{\mathfrak{m}}$ 
  (concretely in the case where the underlying 
  automorphic representation is not CAP, which have reducible
   Galois representations \cite[Thm 2]{Galois-reps}) one only sees contributions from automorphic representations
   $\pi$
   whose 
   components at infinity belong to the discrete series \cite[\S 4.1]{weissauer}.
   Then \cite[Prop 1.5]{Galois-reps} (using that weakly endoscopic representations 
   are Eisenstein) implies that the multiplicity of 
   the (anti)holomorphic discrete series $\pi_{\infty, \text{hol}}$ 
   equals the one of the non-holomorphic discrete series $\pi_{\infty,\text{n-hol}}$ \cite[\S 1.1]{weissauer}. 
   Moreover, the multiplicity of $\pi_{\infty, \text{hol}}$ contributes to 
   coherent cohomology in degree $0$ and $3$, and the multiplicity of $\pi_{\infty, \text{n-hol}}$ 
   contributes to coherent cohomology in degree $1$ and $2$.
  By rescaling an eigenclass to land on
  $\H^{l(w)}(\Sh^{\tor}_{\mathcal{O}_{\overline{\Q}_p}},\omega(w \cdot (a,b)+(3,3)))_{\tilde{\mathfrak{m}}}$ and reducing mod $p$
  we get one implication.  
  For the other direction, 
  by \Cref{lan-suh}(1) $\H^1(\Shbar^{\tor},\omega^{\sub}(a+3,b+3))=0$
  under the restrictions on the weights, so by the second equality we can lift eigensystems in 
  $\H^0(\Sh^{\tor}_{\overline{\F}_p},\omega^{\can}(a+3,b+3))_{\mathfrak{m}}$
  to $\H^0(\Sh^{\tor}_{\check{\Z}_p},\omega^{\can}(a+3,b+3))_{\mathfrak{m}}$, 
  which is moreover torsion-free. Using \eqref{BGG} and reducing mod $p$ implies what we want. 
\end{proof}
\end{prop}

We can now explain how \Cref{geometric-entailment} proves a generic entailment of Serre weights.
Given $\lambda_0=(a,b) \in C_0$ let $\lambda_1 =(p-b-3,p-a-3) \in C_1$ and $\lambda_2=(p+b-1,p-a-3) \in C_2$
be the corresponding affine Weyl reflections. 
\begin{theorem} \label{entailment}
Let $\lambda_0=(a,b) \in C_0$ satisfying $b \ge 1$, $p \ge 5$, and $\mathfrak{m} \subset \mathbb{T}$ a non-Eisenstein generic eigensystem
such that $F(\lambda_0) \in W(\overline{\rho}_{\mathfrak{m}})$.  Then $F(\lambda_1) \in W(\overline{\rho}_{\mathfrak{m}})$ or 
$F(\lambda_2) \in W(\overline{\rho}_{\mathfrak{m}})$.
\begin{proof}
  By \Cref{Weyl-decomposition} we have the exact sequence
  $$
  0 \to F(\lambda_1) \to V(\lambda_1)_{\Fpbar} \to F(\lambda_0) \to 0
  $$ in
  $\Rep_{\overline{\F}_p}\gsp_4(\F_p)$.
 By the concentration of \'etale cohomology in \Cref{coherent-to-betti} we have that
 $\H^3_{\text{\'et}}(\Sh_{\overline{\Q}_p},V(\lambda_1)_{\Fpbar})_{\mathfrak{m}} \neq 0$, 
 which implies that $\H^0(\Shbar,\omega(p-b,p-a))_{\mathfrak{m}} \neq 0$.
 Assuming that $b \ge 1$ 
 the map $\theta^1_{(p-b,p-a)}$ is injective on global sections localized at $\mathfrak{m}$ by \Cref{geometric-entailment},
 so that
 $\H^0(\Shbar,\omega(p+b+2,p-a))_{\mathfrak{m}} \neq 0$. By applying \Cref{coherent-to-betti} again we obtain
 $\H^3_{\text{\'et}}(\Sh_{\overline{\Q}_p},V(\lambda_2)_{\F_p})_{\mathfrak{m}} \neq 0$.
 By \Cref{Weyl-decomposition} 
 $$
 0 \to F(\lambda_2) \to V(\lambda_2)_{\Fpbar} \to F(\lambda_1) \to 0
 $$
 is exact, so we conclude by taking its long exact sequence in \'etale cohomology. 
\end{proof}
\end{theorem}

\subsection{Some non-generic entailments}
We can also prove some non-generic entailments using $H_1$ and $\theta_1$.
\begin{theorem} \label{non-generic}
Let $\mathfrak{m}$ be non-Eisenstein and generic, and $p\ge 5$. Suppose that $F(a,0) \in W(\overline{\rho}_{\mathfrak{m}})$ for some 
$0 \le a \le p-1$. Then 
\begin{enumerate}
\item $F(a+p-1,p-1) \in W(\overline{\rho}_{\mathfrak{m}})$. 
\item If $a \le p-4$ then $F(a+p+1,p-1) \in W(\overline{\rho}_{\mathfrak{m}})$.
\item Let $\lambda_0=(a,b) \in X_1(T)$ such that $a+b=p-3$, $b \neq 0$, and $F(a,b) \in W(\overline{\rho}_{\mathfrak{m}})$. Then 
  $F(2p-a-4,b) \in W(\overline{\rho}_{\mathfrak{m}})$.
\end{enumerate}
\begin{proof}
By \Cref{coherent-to-betti} we have $\H^0(\Shbar,\omega(a+3,3))_{\mathfrak{m}} \neq 0$. Applying $H_{1}$ yields 
$\H^0(\Shbar,\omega(a+p+2,p+2))_{\mathfrak{m}} \neq 0$, so we get $1)$ by the converse direction of \Cref{coherent-to-betti}.
Similarly, we have that $\theta_1$ is injective on $\H^0(\Shbar,\omega(a+3,3))$: if a form is in its kernel it must be 
divisible by $H_1$ by \Cref{divisibility-properties}, and $\H^0(\Shbar,\omega(a-p+4,4-p))_{\mathfrak{m}}=0$ by \Cref{lan-suh}(2). 
Therefore, $\H^0(\Shbar,\omega(a+p+4,p+2))_{\mathfrak{m}} \neq 0$, and we conclude by \Cref{coherent-to-betti}(2).
For $(3)$ we still have that $\theta^1_{(a+3,b+3)}$ is injective 
on $\H^0(\Shbar,\omega(a+3,b+3))$, since we can reduce to injectivity of $\theta_1$ by
\Cref{geometric-entailment}, which follows by \Cref{divisibility-properties} and \Cref{lan-suh}(2).
In this case $V(2p-a-4,b)_{\Fpbar}=F(2p-a-4,b)$ by \Cref{Weyl-decomposition}, which produces the entailment. 
\end{proof}
\end{theorem}

Point $(3)$ is a boundary case of \Cref{entailment}, where $\lambda_0$ is in the boundary of $C_0$ and $C_1$.
One would also expect the entailment $F(a,a) \in W(\overline{\rho}_{\mathfrak{m}}) \implies F(a+p-1,a) \in W(\overline{\rho}_{\mathfrak{m}})$, 
but it seems more difficult to obtain with our techniques, since $\theta_4$ is identically zero on 
$\H^0(\Shbar,\omega(a+3,a+3))$.

\bibliographystyle{alpha}
\bibliography{bib}

\end{document}